\documentclass[a4paper]{article}

\usepackage[english]{babel}
\usepackage[utf8x]{inputenc}
\usepackage[T1]{fontenc}

\usepackage[a4paper,top=2cm,bottom=2.5cm,left=2.3cm,right=2.3cm,marginparwidth=1.75cm]{geometry}

\usepackage[colorinlistoftodos]{todonotes}
\usepackage{hyperref}
\hypersetup{colorlinks=true, allcolors=blue}
\usepackage[all]{xy}
\usepackage[affil-it]{authblk}
\usepackage{amssymb}
\usepackage{cancel}
\usepackage{epstopdf}
\usepackage{verbatim}
\usepackage{amsmath,amscd}
\usepackage{graphicx}
\usepackage{subcaption}
\usepackage{amsthm}
\usepackage{tikz}
\usetikzlibrary{cd}
\usepackage{pgf}
\usepackage{mathrsfs}
\usetikzlibrary{arrows}
\usepackage[toc,page]{appendix}
\usepackage{caption}
\usepackage[colorinlistoftodos]{todonotes}

\newtheorem{theorem}{Theorem}[section]
\newtheorem{prop}[theorem]{Proposition}

\makeatletter
\newtheorem*{rep@theorem}{\rep@title}
\newcommand{\newreptheorem}[2]{%
\newenvironment{rep#1}[1]{%
 \def\rep@title{#2 \ref{##1}}%
 \begin{rep@theorem}}%
 {\end{rep@theorem}}}
\makeatother

 \newreptheorem{theorem}{Theorem}

\newtheorem{lemma}[theorem]{Lemma}

\theoremstyle{definition}
\newtheorem{mydef}[theorem]{Definition}
\newtheorem{rem}[theorem]{Remark}
\newtheorem{exa}[theorem]{Example}
\newtheorem{cor}[theorem]{Corollary}

\title{{\bf{Skeletal filtrations of the fundamental group \\ 
of a non-archimedean curve}}} 
\author{Paul Alexander Helminck}
\affil{
 Swansea University, Department of Mathematics}

\begin{document}
\maketitle

\definecolor{qqqqff}{rgb}{0,0,1}

\begin{abstract}
In this paper we study skeleta of residually tame coverings of a marked curve over a non-archimedean field. We first generalize a result by Liu and Lorenzini by proving a simultaneous semistable reduction theorem for residually tame coverings. 
We then use this to construct a functor from the category of residually tame coverings of a marked curve $(X,D)$ 
to the category of tame coverings of a metrized complex $\Sigma$ associated to $(X,D)$. We enhance the latter category by adding a set of gluing data to every covering and we show that this yields an equivalence of categories. Using this equivalence, we then define filtrations of the fundamental group of the marked curve, giving for instance the absolute decomposition and inertia groups of the metrized complex. We then use the analytic slope formula to prove that the extensions that arise from the abelianizations of the decomposition and inertia quotients coincide with the extensions that arise from the toric and connected parts of the analytic Jacobian of the curve.  
\end{abstract}

\mathchardef\mhyphen="2D
\newcommand\hyp{\mhyphen}

\section{Introduction}

Let $K$ be a complete, algebraically closed non-archimedean field with a non-trivial valuation. In this paper, 
we study tropicalizations of {\it{residually tame}} coverings $\phi:X'\rightarrow{X}$ of curves, which   
 are coverings with an extra tameness condition on their Berkovich analytification. Namely, for every $x'\in{X'^{\mathrm{an}}}$ mapping to $x\in{X^{\mathrm{an}}}$ we require the corresponding 
extension of completed residue fields $\mathcal{H}(x)\subset{\mathcal{H}(x')}$ to be tame. 
 For these coverings, we prove the following simultaneous semistable reduction theorem, which can be seen as a non-discrete generalization of \cite[Theorem 2.3]{liu1}. 
 \begin{theorem}\label{MainThm2}
 Let $(X',D')\to(X,D)$ be a residually tame covering of marked curves. Then the inverse image of any (strongly) semistable vertex set $V$ of $(X,D)$ is a (strongly) semistable vertex set for $(X',D')$. 
 \end{theorem}
We can interpret this theorem categorically as follows. 
We first associate a metrized complex $\Sigma$ to a fixed semistable vertex set $V$ of $(X,D)$ using 
the construction in \cite[Sections 2.16 and 3.22]{ABBR1}. 
This consists of a skeleton of $(X,D)$, a set of residue curves 
for the vertices and an identification of the edges with closed points on the residue curves. 
 In terms of this language, Theorem \ref{MainThm2} then says that there is a natural functor $\mathcal{F}_{\Sigma}$ 
from the category of residually tame coverings of $(X,D)$ to the category of tame coverings of $\Sigma$. 
We denote the first category by $\mathrm{Cov}_{\mathrm{Tame}}(X,D)$ and the second by $\mathrm{Cov}(\Sigma)$. 
One of the main goals in this paper is to investigate what properties 
of $\mathrm{Cov}_{\mathrm{Tame}}(X,D)$ are preserved under $\mathcal{F}_{\Sigma}$. For instance, the category $\mathrm{Cov}_{\mathrm{Tame}}(X,D)$ has the additional structure of a Galois category, which means that there is a profinite fundamental group $\pi_{\mathrm{Tame}}(X,D)$ that classifies residually tame coverings of $(X,D)$ as in classical Galois theory. 
It is then natural to ask whether we can transfer this structure to $\mathrm{Cov}(\Sigma)$ (or an extended category). 
The answer is yes, and we use the lifting results in \cite{ABBR1} to do this. 

\begin{theorem}\label{MainThm3}

Let $V$ be a (strongly) semistable vertex set of $(X,D)$ with skeleton $\Sigma$ and let $\mathcal{F}_{\Sigma}$ 
be the corresponding functor from the category of residually tame coverings of $(X,D)$ to the category $\mathrm{Cov}_{\mathcal{G}}(\Sigma)$ of enhanced tame coverings of $\Sigma$. Then $\mathcal{F}_{\Sigma}$ induces an equivalence of categories
\begin{equation}
\mathrm{Cov}_{\mathrm{Tame}}(X,D)\simeq{{\mathrm{Cov}_{\mathcal{G}}(\Sigma)}}.
\end{equation}
\end{theorem}

The idea of the proof is as follows. 
The results in \cite{ABBR1} show that a tame covering of $\Sigma$ can have multiple (but finitely many) algebraic lifts, and the possible lifts are classified in terms of {\it{gluing data}}. To make the lift unique, we enhance a tame covering with a set of gluing data. We then define morphisms between these enhanced coverings, which gives a category $\mathrm{Cov}_{\mathcal{G}}(\Sigma)$. Using the unique lifting theorem \cite[Theorem 6.18]{ABBR1} for tame coverings of residue curves, we then show that the tropicalization functor $\mathcal{F}_{\Sigma}$ induced by Theorem \ref{MainThm2} gives an equivalence.    

From Theorem \ref{MainThm3}, 
we obtain an induced Galois category structure on $\mathrm{Cov}_{\mathcal{G}}(\Sigma)$ with corresponding profinite fundamental group $\pi(\Sigma)\simeq\pi_{\mathrm{Tame}}(X,D)$.  
We then define filtrations of 
$\pi(\Sigma)$ using the concepts of metrically unramified and completely split coverings. Here a covering $\phi$ is metrically unramified above an edge $e$ of $\Sigma$ if the dilation factors $d_{e'/e}(\phi)$ of the edges $e'$ above $e$ are all $1$, and a covering is completely split above a vertex $v$ of $\Sigma$ if there are $\mathrm{deg}(\phi)$ vertices above $v$. 
For any subcomplex $\Sigma^{0}\subset{\Sigma}$, we then consider all coverings of $\Sigma$ that are metrically unramified or completely split over $\Sigma^{0}$. These coverings correspond to closed normal subgroups of $\pi(\Sigma)$ which we call the absolute inertia group $\mathfrak{I}(\Sigma^{0})$ and the  
absolute decomposition group $\mathfrak{D}(\Sigma^{0})$ respectively. 
Their quotients $\pi_{\mathfrak{I}}(\Sigma^{0}):=\pi(\Sigma)/\mathfrak{I}(\Sigma^{0})$ and $\pi_{\mathfrak{D}}(\Sigma^{0}):=\pi(\Sigma)/\mathfrak{D}(\Sigma^{0})$ classify the connected coverings of $\Sigma$ that are unramified (resp. completely split) above $\Sigma^{0}$. If $D=\emptyset$ and $\Sigma^{0}=\Sigma$,  then we have the following theorem: 


\begin{theorem}\label{FundamentalGroupGraph1}
Let $\mathfrak{D}(\Sigma)$ be the decomposition group of $\Sigma$ in $\pi(\Sigma)$. Then $\pi_{\mathfrak{D}}(\Sigma):=\pi(\Sigma)/\mathfrak{D}(\Sigma)$ is isomorphic to the profinite completion of the ordinary fundamental group of the underlying graph $\Gamma$ of the metrized complex $\Sigma$. 
\end{theorem}

By algebraic topology, it follows that $\pi_{\mathfrak{D}}(\Sigma)$ is isomorphic to the profinite completion of the free group on $\beta(\Sigma)$ generators, where $\beta(\Sigma)$ is the first Betti number of $\Sigma$. 
After this, we turn to the abelianizations of the groups $\pi_{\mathfrak{D}}(\Sigma)$ and $\pi_{\mathfrak{I}}(\Sigma)$. For any $n$ coprime to $\mathrm{char}(k)$, the cyclic \'{e}tale coverings of degree $n$ of an algebraic curve $X$ are classified by the torsion points of the Jacobian $J:=J(X)$ of $X$ using the isomorphism $J[n]\simeq{\mathrm{Hom}(\pi(\Sigma),\mathbf{Z}/n\mathbf{Z})}$. The decomposition and inertia groups  $\pi_{\mathfrak{D}}(\Sigma)$ and $\pi_{\mathfrak{I}}(\Sigma)$ naturally subdivide these cyclic coverings:
\begin{equation}\label{Subdivision1}
\mathrm{Hom}(\pi_{\mathfrak{D}}(\Sigma),\mathbf{Z}/n\mathbf{Z})\subset{\mathrm{Hom}(\pi_{\mathfrak{I}}(\Sigma),\mathbf{Z}/n\mathbf{Z})}\subset{}\mathrm{Hom}(\pi(\Sigma),\mathbf{Z}/n\mathbf{Z}). 
\end{equation} 
On the other hand, we also have such a filtration in the Jacobian of $X$ 
by the results in \cite{Baker2014} and \cite{Bosch1984}. 
Indeed, let $\mathcal{X}$ be a semistable model of $X$ with skeleton $\Sigma$ and let $\mathrm{Jac}(\Sigma)$ be the tropical Jacobian or component group associated to $\Sigma$. We consider the kernel $J^{0}$ of the tropicalization map $\overline{\tau}:J^{\mathrm{an}}\rightarrow{\mathrm{Jac}(\Sigma)}$. 
The reduction $\overline{J}^{0}$ of $J^{0}$ fits in an exact sequence
\begin{equation}\label{Decomposition}
1\rightarrow{}\overline{T}\rightarrow{\overline{J}^{0}}\rightarrow{\overline{B}}\rightarrow{1},
\end{equation}
where $\overline{T}$ is a torus and $\overline{B}$ is an abelian variety over $k$. More explicitly, we have $\overline{B}=\prod_{i}\mathrm{Jac}(\Gamma_{i})$, where the $\Gamma_{i}$ are the components in the special fiber of a semistable model $\mathcal{X}$ for $X$. 
These concepts are generalizations of phenomena in the discretely valued case, where the role of the analytic Jacobian is played by the N\'{e}ron model $\mathcal{J}/R$ of the Jacobian. This N\'{e}ron model $\mathcal{J}$ has a fiberwise connected component $\mathcal{J}^{0}$ and the special fiber of $\mathcal{J}^{0}$ fits into an exact sequence similar to the one in Equation \ref{Decomposition}. 
  

Returning to the analytic side, we now write $J^{0}[n]:=\{P\in{J[n]}:\overline{\tau}(P)=0\}$ and $T[n]:=\{P\in{J^{0}[n]}:\pi(\overline{P})=0\}$, where $\overline{P}$ is the reduction of $P$ and $\pi:\overline{J}^{0}\rightarrow{\overline{B}}$ is the map from Equation \ref{Decomposition}. We then have the inclusions  
\begin{equation}\label{Subdivision2}
T[n]\subset{J^{0}[n]}\subset{J[n]}.
\end{equation} 
We invite the reader to compare this with the material in \cite[Expos\'{e} IX, \S{}12]{grothendieck1972groupes}. We show here that the filtration of $J[n]$ in Equation \ref{Subdivision2} coincides with the filtration in Equation \ref{Subdivision1} under the isomorphism $J[n]\simeq{\mathrm{Hom}(\pi(\Sigma),\mathbf{Z}/n\mathbf{Z})}$.  




 \begin{theorem}\label{TorsionUnramifiedCoverings1}
Let $\pi_{\mathfrak{I}}(\Sigma)$ and $\pi_{\mathfrak{D}}(\Sigma)$ be the inertia and decomposition quotients of $\pi(\Sigma)$ respectively. 
Let $n$ be an integer such that $\mathrm{gcd}(n,\mathrm{char}(k))=1$. Then the isomorphism ${J[n]}\simeq{\mathrm{Hom}(\pi(\Sigma),\mathbf{Z}/n\mathbf{Z})}$ induces isomorphisms 
\begin{equation}
J^{0}[n]\simeq{}{\mathrm{Hom}(\pi_{\mathfrak{I}}(\Sigma),\mathbf{Z}/n\mathbf{Z})}
\end{equation}
and
\begin{equation}
T[n]\simeq{}{\mathrm{Hom}(\pi_{\mathfrak{D}}(\Sigma),\mathbf{Z}/n\mathbf{Z})}.
\end{equation} 
\end{theorem}

The main tool we use in the proof of this theorem is the analytic slope formula \cite[Theorem 5.15]{BPRa1}, 
which says that the local reduced divisor of a function is determined by the slopes of its logarithm. This allows us to relate the various properties of $n$-torsion points in the Jacobian to the covering data on the level of metrized complexes.    
To better illustrate the geometric subtleties behind this theorem, consider the contrapositive of the first result. 
This says that   
a point in $J[n]$ is non-trivial in the component group $\mathrm{Jac}(\Sigma)$  
if and only if  
the corresponding covering is {{metrically ramified}} above {{at least}} one non-trivial edge in the skeleton. This condition is sharp, in the sense that there are examples of cyclic coverings that are metrically ramified on a strict subgraph of $\Sigma$.  
 %
%
Overall, we interpret Theorem \ref{TorsionUnramifiedCoverings1} as giving geometric interpretations for the various torsion points in the Jacobian in terms of coverings. 
In line with this thought it now also seems natural to view 
 the inertia and decomposition groups as non-abelian generalizations of the connected and toric parts of the Jacobian.\footnote{The terminology for the toric part is borrowed from \cite[Expos\'{e} IX, \S{12}]{grothendieck1972groupes}, the connected part is called the fixed part there. } 


The paper is organized as follows. We start in Section \ref{Preliminaries} by proving some results on residually tame morphisms. We then review the results in \cite{ABBR1} and \cite{Baker2014} on lifting morphisms of metrized complexes and on skeleta of Jacobians.  
In Section \ref{SectionCoverings}, we prove the generalized simultaneous semistable reduction theorem for residually tame coverings. 
We then study the notion of an {\it{enhanced morphism}} of metrized complexes in Section \ref{RigidifiedCoverings}. This gives rise to a category and we prove the equivalence of categories stated in Theorem \ref{MainThm3}. In Section \ref{TropicalFundamental}, we reap the benefits of this equivalence and study the fundamental group of the marked curve $(X,D)$ through its tropicalization.  
In Section \ref{Abelianization1} we study the abelianization of $\pi(\Sigma)$  and give a proof of Theorem \ref{TorsionUnramifiedCoverings1}.  

This paper uses a great deal of concepts and results from \cite{ABBR1} and \cite{Baker2014}. To make the transition somewhat easier, we chose to adopt their notation for the most part (see Section \ref{OverviewSection} 
 for some exceptions). As such, the reader might benefit from a review of those papers, and more specifically of the following parts: \cite[Sections 2,4,6 and 7]{ABBR1} and \cite[Sections 4,5, and 6]{Baker2014}.

\subsection{Connections}

We give an overview here of other similar results in the literature. The notion of fundamental group in the context of non-archimedean spaces is well studied, see for instance \cite{DeJong1995} for fundamental groups of Berkovich spaces in terms of \'{e}tale coverings and \cite{Andre2003} for tempered fundamental groups. In \cite{DeJong1995}, one starts with the ordinary fundamental group for algebraic coverings and then moves on to the so-called topological coverings (we call these completely split coverings). Our approach lies in between these two, since every topological covering is residually tame (see Proposition \ref{CategoryToric}), but not every algebraic covering is residually tame (see Example \ref{SeparableExample}). 
The statement of Theorem 
\ref{FundamentalGroupGraph1} 
 appears in various guises throughout the literature, see for instance 
 \cite[Theorem 2.6]{DeJong1995}.  Our results can be seen as a natural generalization where the category of coverings of a graph is replaced by the more general category of enhanced tame coverings of a metrized complex $\Sigma$.
 
 A paper in this area that is also close to ours is \cite{SADI1997}. Here one creates an auxiliary fundamental group using the intersection graph of a semistable model $\mathcal{X}/R$ over a discrete valuation ring $R$ and one then shows (under some restrictions) that its profinite completion is isomorphic to the \'{e}tale fundamental group of an open subscheme of $\mathcal{X}$. Since the base ring is fixed, this 
 excludes a large part of what we call {\it{metrically ramified}} morphisms, essentially by \cite[Th\'{e}or{e}me 3.7]{SADI1997}. In particular, one does not consider all extensions that arise from the component group of the N\'{e}ron model of the Jacobian of the curve. Another advantage that this paper has is that the isomorphisms of profinite fundamental groups that we create are natural, in the sense that they arise from an equivalence of Galois categories. This stronger result for instance implies that there is a unique Galois closure for an enhanced covering of metrized complexes. 
 
 
 
 We can also compare the profinite fundamental groups studied here to the 
 tempered fundamental groups in \cite{Andre2003}.    
 There, one considers quotients of composites of finite \'{e}tale coverings and possibly infinite topological coverings. The corresponding fundamental groups often have a non-profinite part, for instance if an elliptic curve has multiplicative reduction then 
\begin{equation}
\pi^{\mathrm{temp}}(E)\simeq{\hat{\mathbf{Z}}\times{\mathbf{Z}}}.
\end{equation} 
We invite the reader to compare this to Example \ref{EllipticCurves}. A paper in this area that is close on certain levels to ours 
is \cite{Lepage2010}, where it is shown that the metric structure of the skeleton $\Sigma$ of a Mumford curve $X$ can be recovered from its tempered fundamental group for fields of mixed characteristic. To do this, one studies $\mathbf{Z}/p^{n}\mathbf{Z}$-torsors of $X$ using the universal topological covering space $\Omega$ of $X$ and pullbacks along theta functions $\theta:\mathbf{G}_{m}\to\Omega$. 
Our study of these torsors is different in several ways. For instance, we do not restrict ourselves to Mumford curves, relying instead on the finer structure of the analytic Jacobian studied in \cite{Baker2014}. This then gives a relation with the component group or tropical Jacobian of the curve, a notion which does not seem to play any role in  \cite{Lepage2010}.  On the other hand, the torsors that are used in  \cite{Lepage2010} are more general in the sense that they allow wild ramification, whereas the ones we treat here belong to the milder class of residually tame coverings. 





A concept in this paper that seems new is the notion of a decomposition and inertia group of a subcomplex $\Sigma^{0}\subseteq{\Sigma}$. For a non-marked curve $X$ with $\Sigma^{0}=\Sigma$, the decomposition group can be seen as the closed subgroup of $\pi(\Sigma)$ corresponding to the topological coverings. 
We define more general inertia and decomposition groups in Section \ref{FundGroupSection} using the concepts of unramified and completely split coverings. Moreover, we show that there is a connection between the abelianizations of the groups $\pi(\Sigma)/\mathfrak{D}(\Sigma)$ and $\pi(\Sigma)/\mathfrak{I}(\Sigma)$ to the toric and connected parts of the Jacobian in Theorem \ref{TorsionUnramifiedCoverings1}. This result also seems new.


The simultaneous semistable reduction theorem we prove in Theorem \ref{MainThm2} is a generalization of \cite[Theorem 2.3]{liu1}. There, the covering $X'\rightarrow{X}$ is assumed to be Galois and the Galois group $G$ satisfies the tameness condition $p\nmid{|G|}$. In terms of this paper, we say that the covering is Galois-topologically tame, see Section \ref{AlgebraicAnalyticPrelim}. 
A possible proof for the Galois-topologically tame version of \ref{MainThm2} was suggested in \cite[Section 1.2.1]{Cohen2016} using results by Berkovich. 
We combine these results with the lifting results in \cite{ABBR1} to give a full proof in the more general residually tame case. We note however that this uses the semistable reduction theorem, so this does not give a stand-alone proof of this theorem. 
A proof of the analogous theorem in the discretely valued case can be found in \cite[Theorem 1.1]{NewtonPuiseuxSemSta}. The proof of this theorem does not rely on the semistable reduction theorem.    

\section{Preliminaries}\label{Preliminaries}
In this section we give a summary of the algebraic and analytic results that we will need in the rest of the paper. We start by studying the concept of residual tameness and Galois coverings of analytic spaces. 
In Section \ref{OverviewSection} we review the lifting results in \cite{ABBR1} and we point out what adjustments have to be made. 

\subsection{Topological and residual tameness}\label{AlgebraicAnalyticPrelim}

We will use the following notation throughout this paper:
\begin{enumerate}
\item $K$ is a complete, algebraically closed non-archimedean field with non-trivial valuation $\mathrm{val}:K\rightarrow{\mathbf{R}\cup\{\infty\}}$ and value group $\Lambda:=\mathrm{val}(K^{*})$,
\item $R$ is its valuation ring,
\item $\mathfrak{m}_{R}$ is its maximal ideal,
\item $k=R/\mathfrak{m}_{R}$ is its residue field and
\item $\varpi$ is any element in $K$ with $\mathrm{val}({\varpi})>0$. 
\end{enumerate}

We use $p$ to denote the characteristic of the residue field. Throughout the paper, we will impose tameness conditions of the form "$p\nmid{n}$" or "$n$ is coprime to $p$", for $n$ an integer. For $p=0$ these are void conditions.  
A curve $X$ over $K$ is a smooth  
proper scheme of finite type over $K$ of dimension $1$. We will also refer to a curve as an algebraic curve. 
A marked curve $(X,D)$ is a curve $X/K$ together with a finite set 
$D\subset{X(K)}$. 
A morphism of curves is a finite separable morphism $\phi:X'\rightarrow{X}$ of schemes over $K$. 
A morphism of marked curves $(X',D')\rightarrow{(X,D)}$ is a morphism of curves that is \'{e}tale on $X'\backslash{D'}$, where $\phi^{-1}(D)=D'$. Instead of morphisms, we will also call these coverings.  
A semistable model for a curve $X/K$ is a triplet $(\mathcal{X},\pi,\psi)$, where $\mathcal{X}$ is a scheme,  
 $\pi:\mathcal{X}\rightarrow{\mathrm{Spec}(R)}$ is a flat, proper morphism such that the special fiber $\mathcal{X}_{s}$ is reduced and only has ordinary double singularities, and $\psi:\mathcal{X}_{\eta}\simeq{X}$ is an isomorphism (here $\mathcal{X}_{\eta}$ is the generic fiber of $\mathcal{X}$).   
We will denote a semistable model by $\mathcal{X}$. A strongly semistable model for $X$ is a semistable model such that the irreducible components of the special fiber are smooth. 

An analytic space over $K$ is as in \cite{Berkovich1993} and \cite{Temkin2}. We will only be needing good analytic spaces, which are analytic spaces such that every point has an affinoid neighborhood. For these spaces, 
we can freely use 
the material in \cite{berkovich2012} and \cite{Berkovich1993}. We will use the notation and definitions adopted in \cite[Sections 2-4]{BPRa1} for annuli, disks (which are called balls there), semistable vertex sets and skeleta. 

Let $L$ be a complete non-archimedean field over $K$. To introduce the notion of $L$-valued points of an analytic space over $K$, we use the 
ground field extension functor. 
This sends an analytic space $U$ over $K$ to the analytic space $U_{L}:=U\hat{\otimes}_{K}{L}$ over $L$. An $L$-valued point of $U$ is then a morphism $\mathcal{M}(L)\to{U_{L}}$ of $L$-analytic spaces. We denote the set of all such $L$-valued points by $U(L)$. 
For a point $x\in{U}$, we denote the completed residue field by $\mathcal{H}(x)$. 
Using the following lemma, we can equivalently view $L$-valued points of $U$ as points $x\in{U}$ together with a bounded $K$-homomorphism $\mathcal{H}(x)\to{L}$. We leave its proof to the reader, see \cite[Section 1.4]{Berkovich1993} and \cite[Section 4.1.6]{Temkin2} for more information.    


\begin{lemma}\label{ValuedPoints}
Let $L$ be a complete non-archimedean field over $K$ and let $U$ be an analytic space over $K$. Then there is a bijection between the set $U(L)$ and 
the set of points $x\in{U}$ together with a bounded $K$-homomorphism $\mathcal{H}(x)\to{L}$.\footnote{Note that this morphism is bounded if and only if it defines an isometric embedding.}
\end{lemma}

For any morphism $U\to{U'}$ of $K$-analytic spaces, we obtain an induced morphism 
\begin{equation}
U\hat{\otimes}_{\mathcal{M}(K)}\mathcal{M}(L)\to{U'\hat{\otimes}_{\mathcal{M}(K)}\mathcal{M}(L)}
\end{equation} of $L$-analytic spaces. By composing morphisms, we then obtain a map $U(L)\to{U'(L)}$. 
If we fix an analytic space $U$ and allow $L$ to vary, then this induces a functor from the  
category of complete non-archimedean field extensions of $k$ to the category of sets. In terms of this language, a morphism $U\to{U'}$ induces a natural transformation of the corresponding functors.      

For any algebraic curve $X$ as above, we denote its Berkovich analytification in the sense of \cite{berkovich2012} by $X^{\mathrm{an}}$. 
For any covering 
$\phi:X'\rightarrow{X}$ of curves over $K$, we obtain from the analytification functor a finite morphism of analytic spaces $\phi^{\mathrm{an}}:X'^{\mathrm{an}}\rightarrow{X^{\mathrm{an}}}$.  

\begin{mydef}{\it{(Topological and residual tameness)}}\label{AnalyticallyTame}
Let $X$ be a curve over $K$ and let $\phi:X'\rightarrow{X}$ be a morphism of curves with analytification $\phi^{\mathrm{an}}:X'^{\mathrm{an}}\rightarrow{X^{\mathrm{an}}}$. 
Let $x'\in{X^{\mathrm{an}}}$ with $\phi^{\mathrm{an}}(x')=x$  
and consider the extension of completed residue fields
\begin{equation}
i_{x'}:\mathcal{H}(x)\rightarrow{\mathcal{H}(x')}.
\end{equation}
We say that 
\begin{enumerate}
\item $\phi^{\mathrm{an}}$ is residually tame at $x'$ if $i_{x'}$ is a tame extension of valued fields.  
\item $\phi^{\mathrm{an}}$ is topologically tame at $x'$ if $[\mathcal{H}(x'):\mathcal{H}(x)]$ is coprime to $p:=\mathrm{char}(k)$. 
\end{enumerate}
We say that $\phi$ or $\phi^{\mathrm{an}}$ is {{topologically tame}} (resp. residually tame) if $\phi^{\mathrm{an}}$ is topologically tame (resp. residually tame) at every point of $X'^{\mathrm{an}}$. A morphism of marked curves $(X',D')\to(X,D)$ is topologically (resp. residually) tame if the underlying morphism $X'\to{X}$ has the corresponding property.   
\end{mydef}
These concepts can also be found in \cite{TEMKIN2017}, \cite{Cohen2016} and \cite{Berkovich1993}.\footnote{A tame extension of valued fields is called a moderately ramified extension in \cite{Berkovich1993}.} 
If $\phi^{\mathrm{an}}$ is topologically tame at a point $x'$, then it is also residually tame at that point, but the converse is not necessarily true, see Example \ref{SeparableExample}. 
For any point $x'$ at which $\phi^{\mathrm{an}}$ is residually tame and \'{e}tale, we have the following relation by \cite[Proposition 2.4.7]{Berkovich1993}:  
\begin{equation}\label{IndicesResidueFields}
[\mathcal{H}(x'):\mathcal{H}(x)]=[\tilde{\mathcal{H}}(x'):\tilde{\mathcal{H}}(x)][|\mathcal{H}(x')^{*}|:|\mathcal{H}(x)^{*}|]
\end{equation}
The morphism $\phi^{\mathrm{an}}$ is then topologically tame at $x'$ if both of the indices in Equation \ref{IndicesResidueFields} are not divisible by $\mathrm{char}(k)$. 

\begin{exa}\label{SeparableExample}
We give an example of a covering of curves that is residually tame and \'{e}tale, but not topologically tame. Consider an elliptic curve $E$ over $\mathbf{C}_{2}$ with ordinary reduction, meaning that there is a model $\mathcal{E}/R$ with good reduction such that $\mathcal{E}[2](\overline{\mathbf{F}}_{2})=\mathbf{Z}/2\mathbf{Z}$. To be more explicit, consider the elliptic curve over $\mathbf{C}_{2}$ given by the equation
\begin{equation}\label{OrdinaryCurve}
y^2+xy=x^3+1.
\end{equation} 
The projective homogenization of this equation in $\mathbf{P}^{2}_{R}$ gives a model $\mathcal{E}/R$ with the desired reduction. The non-trivial $2$-torsion point in the special fiber is given in local coordinates by $\overline{P}=(0,1)$.

We denote the special fiber of $\mathcal{E}$ by $\overline{E}$. 
Since $\overline{E}$ has a non-trivial $2$-torsion point, there is a unique \'{e}tale morphism $\overline{E}'\rightarrow{\overline{E}}$ of degree two up to $\overline{E}$-isomorphism. More explicitly, it is the isogeny dual to the Frobenius morphism. Using \cite[Theorem 7.4]{ABBR1}, we see that there is a lift of the morphism $\overline{E}'\rightarrow{\overline{E}}$ to a finite 
morphism of semistable models $\mathcal{E}'\rightarrow{\mathcal{E}}$. The generic points $\eta'$ and $\eta$ of the special fiber of $\mathcal{E}'$ and $\mathcal{E}$ correspond to type-$2$ points $x'$ and $x$ of the Berkovich spaces $E'^{\mathrm{an}}$ and $E^{\mathrm{an}}$. The induced morphism $\tilde{\mathcal{H}}(x)\rightarrow{\tilde{\mathcal{H}}(x')}$ is just the map of function fields $\overline{\mathbf{F}}_{2}(E)\rightarrow{\overline{\mathbf{F}}_{2}(E')}$, which is separable by construction. We thus see that the morphism $E'^{\mathrm{an}}\rightarrow{E^{\mathrm{an}}}$ is residually tame over $x$. The morphism splits completely over the other points (which can be seen using \cite[Theorem 4.35]{ABBR1}), so it defines a residually tame covering. Note that it is not topologically tame, since $[\mathcal{H}(x'):\mathcal{H}(x)]=2$.

\end{exa}

\begin{rem}
We note one important difference between the usual notion of tameness for non-archimedean fields and the notion of topological tameness. For a finite separable extension of complete non-archimedean fields $K\subset{L}$, the tameness of $L$ implies the tameness of the Galois closure $\overline{L}$ over $K$. The same is not necessarily true for topological tameness. To see this, we first note that 
the factor $[|\mathcal{H}(x')^{*}|:|\mathcal{H}(x)^{*}|]$ will not be divisible by $\mathrm{char}(k)=p$ after passing to the Galois closure, see \cite[Section 2.4]{Berkovich1993}. However, the degree $[\tilde{\mathcal{H}}(x'):\tilde{\mathcal{H}}(x)]$ of the extension of reductions can easily become divisible by $p$ after passing to the Galois closure. In the proof of Theorem \ref{MainThm2} we will show that in certain important cases the residual tameness does imply the topological tameness of the Galois closure.     
\end{rem}

\begin{mydef}\label{DefinitionGResTame}
{\it{(Galois-topologically tame coverings)}} Let $X'\rightarrow{X}$ be a covering of connected curves with extension of function fields $K(X)\rightarrow{K(X')}$. Let $K(\overline{X})$ be the Galois closure of this extension and let $\overline{X}\rightarrow{X}$ be the corresponding covering of connected curves. We say that $X'\rightarrow{X}$ is Galois-topologically tame if $\overline{X}\rightarrow{X}$ is topologically tame.  
\end{mydef}

Note that this definition was also used in \cite[Section 6.3]{Berkovich1993} to compare coverings of algebraic curves and coverings of Berkovich spaces. It also implicitly plays a role in 
\cite[Theorem 2.3]{liu1}, where the coverings are Galois with Galois group $G$ and $p\nmid{|G|}$. 


\vspace{0.3cm}

We now study the process of moving from an algebraic Galois covering to an analytic Galois covering. Our main goal here is to give an explicit description of the completed residue field $\mathcal{H}(\overline{x})$ for a point $\overline{x}\in{\overline{X}^{\mathrm{an}}}$, where $\overline{X}$ is the Galois closure of a covering $X'\rightarrow{X}$. The answer will be given in Proposition \ref{GaloisClosureCriterion}. We start with some preliminaries on group actions. 

\begin{mydef}{\it{(Galois coverings of analytic spaces)}} \label{GaloisCovering}
Let $U$ be an analytic space over $K$ 
and let $\psi:V\rightarrow{U}$ be a finite \'{e}tale morphism. We say that $\psi$ is a Galois covering if there exists a finite group $G$ and a homomorphism $G\rightarrow{\mathrm{Aut}_{U}(V)}$ such that the induced map 
$G\times_{U}V\rightarrow{V\times_{U}V}$ is an isomorphism. Here we consider $G$ as a group object over $U$, so it is a disjoint union of $|G|$ copies of $U$. \end{mydef}

Galois coverings of analytic spaces arise naturally from coverings of algebraic curves, as we now explain. 
Let $\overline{\phi}:\overline{X}\rightarrow{X}$ be a covering of connected curves and suppose that the corresponding extension of function fields $K(X)\rightarrow{K(\overline{X})}$ is Galois. 
Write $U=X\backslash{D}$, where $D$ is the branch locus of $\overline{\phi}$ and $\overline{U}=\overline{X}\backslash{D'}$ for $D'=\phi^{-1}(D)$. The morphism $\overline{U}\rightarrow{U}$ is then \'{e}tale, which implies that 
$\overline{U}^{\mathrm{an}}\rightarrow{U^{\mathrm{an}}}$ is \'{e}tale. Furthermore, we have an action $G\times_{U}{\overline{U}}\rightarrow{\overline{U}}$ and $\overline{U}$ is a $G$-torsor under this action. That is, it induces an isomorphism  
\begin{equation}
G\times_{U}{\overline{U}}\rightarrow{\overline{U}\times_{U}{\overline{U}}}.
\end{equation} 
This implies the analytification $G\times_{U^{\mathrm{an}}}{\overline{U}}^{\mathrm{an}}\rightarrow{\overline{U}^{\mathrm{an}}\times_{U^{\mathrm{an}}}\overline{U}^{\mathrm{an}}}$ is also an isomorphism (here we again write $G$ for the constant group object). It follows that $\overline{U}^{\mathrm{an}}\rightarrow{U^{\mathrm{an}}}$ is a Galois covering in the sense of Definition \ref{GaloisCovering}. Now consider a point $x\in{U^{\mathrm{an}}}$. This defines a morphism $\mathcal{M}(\mathcal{H}(x))\rightarrow{U^{\mathrm{an}}}$ and the fiber product $Z:=\mathcal{M}(\mathcal{H}(x))\times_{U^{\mathrm{an}}}{\overline{U}^{\mathrm{an}}}$ can be identified with the Berkovich spectrum of a finite $\mathcal{H}(x)$-algebra $\mathcal{A}$. 
 We now have an induced action $G\times_{\mathcal{H}(x)}Z\rightarrow{Z}$ and this gives an isomorphism
\begin{equation}
G\times_{\mathcal{H}(x)}Z\rightarrow{Z\times_{\mathcal{H}(x)}Z}.
\end{equation}  
Viewing this as a functor (see Lemma \ref{ValuedPoints}), this says that the action of $G$ on every geometric point of $Z$ is simply transitive. 
By the anti-equivalence between $K$-affinoid algebras and $K$-affinoid Berkovich analytic spaces (see \cite[3.3.8]{Temkin2}), this quickly implies that the algebra $\mathcal{A}$ has $|G|$ distinct $\mathcal{H}(x)$-automorphisms. This number is equal to the dimension of this $\mathcal{H}(x)$-vector space, since it is the degree. We thus obtain 
$\mathcal{A}^{\mathrm{G}}=\mathcal{H}(x)$. 
Using the fact that the action is simply transitive on the geometric points of $Z$,  
we then easily conclude that every individual factor $\mathcal{H}(\overline{x})$ of $\mathcal{A}$ is Galois over $\mathcal{H}(x)$ with Galois group the decomposition group $D_{\overline{x}/x}=\{\sigma\in{G}:\sigma(\overline{x})=\overline{x}\}$ of 
$\overline{x}$ over $x$. 

\begin{prop}\label{GaloisClosureCriterion}
Let $\phi:X'\rightarrow{X}$ be a morphism of connected curves with Galois closure $\overline{X}\rightarrow{X}$ and consider a point $x\in{X^{\mathrm{an}}}$ over which $\overline{\phi}^{\mathrm{an}}$ is \'{e}tale. We denote its completed residue field by $\mathcal{H}(x)$. Let $x_{i}$ be the points in $X'^{\mathrm{an}}$ lying over $x$ and let $\overline{x}$ be a point in $\overline{X}^{\mathrm{an}}$ lying over $x$. The extension $\mathcal{H}(\overline{x})\supset{\mathcal{H}(x)}$ is then Galois. The composite of the $\mathcal{H}(x)$-embeddings  
\begin{equation}
\mathcal{H}(x_{i})\rightarrow{\mathcal{H}(\overline{x})}
\end{equation} 
arising from the transitive action on the fiber over $x$ is 
 $\mathcal{H}(\overline{x})$. 
\end{prop} 
\begin{proof}
Given the earlier-mentioned results on the $\mathcal{H}(x)$-algebra $\mathcal{A}$, the verification of the proposition is an exercise in Galois theory (but in the context of separable $\mathcal{H}(x)$-algebras). We leave the details to the reader. 

\end{proof}

\begin{cor}
\label{GaloisClosureCorollary}
Consider the set-up as in Proposition \ref{GaloisClosureCriterion} and let $\overline{\mathcal{H}}(x_{i})$ be a Galois closure of $\mathcal{H}(x_{i})$ over $\mathcal{H}(x)$. Suppose that $p\nmid{[\overline{\mathcal{H}}(x_{i}):\mathcal{H}(x)]}$ for every $i$. Then $p\nmid[\mathcal{H}(\overline{x}):\mathcal{H}(x)]$. 
\end{cor}
\begin{proof}
This follows directly from Proposition \ref{GaloisClosureCriterion} and the fact that the Galois group of a composite of Galois extensions is a quotient of the direct product of the individual Galois groups. 
\end{proof}

\begin{cor}\label{CorollaryResidualGalois}
Consider the set-up as in Proposition \ref{GaloisClosureCriterion} and suppose that $\phi$ is residually tame. Then the Galois closure $\overline{\phi}$ is residually tame. 
\end{cor}
\begin{proof}
By \cite[Section 2.4]{Berkovich1993}, we find that if $\mathcal{H}(x_{i})\supset{\mathcal{H}(x)}$ is tame, then the Galois closure is also tame and any composite of tame extensions is again tame. By Proposition \ref{GaloisClosureCriterion}, we then see that $\mathcal{H}(\overline{x})$ is also tame for any point $\overline{x}$ lying over $x$ in the Galois closure $\overline{X}$ of $X'$ over $X$. 
\end{proof}


\begin{lemma}\label{CompositeTame}
Suppose that $K(X)\subset{K(X_{1})}\subset{\overline{K(X)}}$ and $K(X)\subset{K(X_{2})}\subset{\overline{K(X)}}$ are function field extensions in a fixed algebraic closure $\overline{K(X)}$ of $K(X)$ corresponding to residually tame (resp. Galois-topologically tame) coverings of $(X,D)$. Consider their composite $M=K(X_{1})\cdot{K(X_{2})}$. 
Then $M$ is residually tame (resp. Galois-topologically tame).
\end{lemma}
\begin{proof}
By Corollaries \ref{GaloisClosureCorollary} and \ref{CorollaryResidualGalois}, the Galois closures of both $K(X_{i})/K(X)$ are residually tame (resp. Galois-topologically tame). It suffices now to show that the composite of the Galois closures has the desired property. We thus assume that the $X_{i}$ are Galois, and we write $\overline{X}$ for the curve corresponding to the composite $M$ of the $K(X_{i})$. Let $\overline{x}\in\overline{X}^{\mathrm{an}}$ be a point in the \'{e}tale locus and consider its image $x_{i}\in{X^{\mathrm{an}}_{i}}$. Writing $H_{i}$ for $\mathrm{Gal}(M/K(X_{i}))$, we then have that $\mathcal{H}(x_{i})\subset{\mathcal{H}(\overline{x})}$ is Galois, with Galois group $D_{\overline{x}/x}\cap{H_{i}}$. 
The composite of the $\mathcal{H}(x_{i})$ under these embeddings corresponds to the intersection of these subgroups $D_{\overline{x}}\cap{H_{i}}$. But we have $H_{1}\cap{H_{2}}=(1)$ since $M$ is the composite of the $K(X_{i})$, so we find that $\mathcal{H}(\overline{x})$ is the composite of the $\mathcal{H}(x_{i})$. Furthermore, tame extensions are stable under taking composites, 
so we conclude that $\mathcal{H}(\overline{x})\supset{\mathcal{H}(x)}$ is also residually tame.  Since $\overline{x}$ was arbitrary, we conclude that $\overline{X}\rightarrow{X}$ is residually tame. Similarly, for Galois-topologically tame coverings, we conclude that $[\mathcal{H}(\overline{x}):\mathcal{H}(x)]$ is not divisible by $p$, because $\mathcal{H}(\overline{x})$ is the composite of two Galois extensions $\mathcal{H}(x_{i})\supset{\mathcal{H}(x)}$ with $[\mathcal{H}(x_{i}):\mathcal{H}(x)]$ not divisible by $p$. 
\end{proof}

We now consider the composite $M_{\mathrm{Tame}}$ (resp. $M_{\mathrm{GTop}}$) of all function fields $K(X')$ corresponding to residually tame (resp. Galois-topologically tame) coverings $(X',D')\rightarrow{(X,D)}$ of a fixed marked curve $(X,D)$. 
Using Corollary \ref{CorollaryResidualGalois} and Lemma \ref{CompositeTame}, we easily see that $M_{\mathrm{Tame}}$ and $M_{\mathrm{GTop}}$ are Galois. 

\begin{lemma}\label{ResiduallyTameLemma}
Let $K(X)$ be the function field of a curve $X$ with a fixed set of marked points $D\subset{X(K)}$. 
Consider the composite $M_{\mathrm{Tame}}$ (resp. $M_{\mathrm{GTop}}$) of all function field extensions $K(X')\supset{K(X)}$ arising from residually tame (resp. Galois-topologically tame) coverings $(X',D')\rightarrow{(X,D)}$. 
Then $M_{\mathrm{Tame}}$ and $M_{\mathrm{GTop}}$ are Galois. 
\end{lemma}
\begin{proof}
Let $x$ be an element of $M_{\mathrm{Tame}}$ or $M_{\mathrm{GTop}}$ and let $L$ be the field generated by $x$ over $K(X)$. By Lemma \ref{CompositeTame}, $L$ is residually tame (resp. Galois-topologically tame) over $K(X)$. 
By Corollary \ref{CorollaryResidualGalois}, we then see that the Galois closure $\overline{L}$ over $K$ is residually tame. For Galois-topologically tame coverings, $\overline{L}$ is automatically topologically tame. The extensions $M_{\mathrm{Tame}}$ and $M_{\mathrm{GTop}}$ are thus Galois. 
\end{proof}

We denote the Galois groups of the extensions in Lemma \ref{ResiduallyTameLemma} by $\pi_{\mathrm{Tame}}(X,D)$ and $\pi_{\mathrm{GTop}}(X,D)$. We view these as the fundamental groups of suitable Galois categories. For more background regarding the notion of a Galois category, we refer the reader to \cite[\href{https://stacks.math.columbia.edu/tag/0BMQ}{Tag 0BMQ}]{stacks-project}, \cite{SGA1},  \cite{Lenstra} and \cite{Cadoret2013}.  

\begin{mydef}\label{GaloisCategoryResiduallyTame1}
{\it{(Galois-topologically and residually tame coverings)}} 
Let $X/K$ be a connected curve with a set of marked points $D$. 
We write $\mathrm{Cov}(X,D)$ for the category of all finite \'{e}tale coverings of $X\backslash{D}$, or equivalently, 
the category of coverings of the marked curve $(X,D)$. 
The full subcategories of $\mathrm{Cov}(X,D)$ of all coverings that are residually (resp. Galois-topologically) tame are denoted by $\mathrm{Cov}_{\mathrm{Tame}}(X,D)$ and $\mathrm{Cov}_{\mathrm{GTop}}(X,D)$ respectively. By Lemma \ref{ResiduallyTameLemma}, these are Galois categories with profinite fundamental groups $\pi_{\mathrm{Tame}}(X,D)$ and $\pi_{\mathrm{GTop}}(X,D)$.    
\end{mydef}


\begin{rem}
Throughout this paper, we suppress the base-points in the profinite fundamental groups $\pi(X,x)$ since they are not of any significant relevance to us.  
\end{rem}

\subsection{The lifting results in \cite{ABBR1}}\label{OverviewSection}

In this section, we point out how the notions in 
\cite{ABBR1} have to be modified to allow for disconnected metrized complexes and morphisms thereof. We leave it to the reader to write out their formal definitions. 
The notions of $\Lambda$-metric graphs, augmented metric graphs and metrized complexes from \cite{ABBR1} transfer directly to the disconnected case. 
For morphisms, we do the following. We start with a continuous map of $\Lambda$-metric graphs $\phi:\Sigma'\rightarrow\Sigma$, where $\Sigma'$ and $\Sigma$ can now be disconnected. Since the image of a connected subset is again connected, we have induced maps 
\begin{equation}
\phi_{i,j}:\Sigma'_{i,j}\rightarrow\Sigma_{i},
\end{equation}   
where the $\Sigma$ are the connected components of $\Sigma$ and similarly for $\Sigma'$. The definitions in  \cite{ABBR1} can now be directly generalized using the $\phi_{i,j}$, which are maps of connected $\Lambda$-metric graphs. 
Note that a harmonic morphism of $\Lambda$-metric graphs is not necessarily surjective anymore, but if we restrict to the connected components of $\Sigma$ that are in the image, then it is again automatically surjective. 

\begin{rem}
Throughout this paper, we will assume that our metrized complexes $\Sigma$ are {\it{loopless}} unless mentioned otherwise. That is, $\Sigma$ contains no edges whose closure is homeomorphic to a circle. The most important exception to this assumption is in Section \ref{SectionCoverings}, where the skeleta are not  assumed to be loopless (or equivalently, the semistable vertex sets are not necessarily strongly semistable). This gives a slightly more general tame simultaneous semistable reduction theorem. Note that for a strongly semistable model $\mathcal{X}$, we automatically have that the corresponding metrized complex $\Sigma(\mathcal{X})$ is loopless. 
\end{rem}
Suppose now that we have a possibly disconnected curve $X/K$ with connected components $X_{i}$. We have  
$$X^{\mathrm{an}}=(\coprod{X_{i}})^{\mathrm{an}}=\coprod{(X_{i})^{\mathrm{an}}}.$$
We can then introduce the notions of semistable vertex sets, triangulated marked curves (which are called triangulated punctured curves in \cite{ABBR1}) and skeleta using their corresponding counterparts from the connected case. Morphisms of these objects are then similarly transported to the disconnected case. The gluing data introduced in \cite[Section 7]{ABBR1} is generalized by taking products of the gluing data on the connected components. We denote this gluing data by $\mathcal{G}(\Sigma',X)$, where $\Sigma'\rightarrow{\Sigma}$ is a map of metrized complexes of $k$-curves and $\Sigma$ is a fixed skeleton of $(X,{D})$. 

\begin{rem}\label{IdentificationGluingData547}
By definition, a set of gluing data is given by isomorphisms 
$$\tau^{-1}_{x'}(e'^{o})\to\tau^{-1}_{y'}(e'^{o}),$$
see \cite[Section 7]{ABBR1}. This however gives us the possibility to compose any such isomorphism with an automorphism of the annulus in question. For instance, for any $u\in{R^{*}}$, we obtain a non-trivial automorphism by multiplying a parameter $t$ by $u$.  The automorphisms in \cite[Lemma 6.10]{ABBR1} are an example of this with $u=\zeta_{n}$ a primitive root of unity. We now see that the annuli $\tau^{-1}_{x'}(e'^{o})$ and $\tau^{-1}_{y'}(e'^{o})$ afford infinitely many distinct automorphisms and this implies that the set of gluing data is infinite. This seems to contradict the results in \cite{ABBR1}, so we instead require these isomorphisms to be $\tau^{-1}(e^{o})$-isomorphisms (as is already done implicitly on several occasions in \cite{ABBR1}, see Example 7.8 where the set of gluing data is of order four). For these $\tau^{-1}(e^{o})$-isomorphisms, the gluing data is finite, giving $d_{e'/e}(\phi)$ possibilities per edge. 
We now see that we can identify the set of gluing data $\mathcal{G}(\Sigma',X)$ with the set $\prod_{e'\in{E(\Sigma')}}\mathbf{Z}/d_{e'/e}(\phi)\mathbf{Z}$. This will be used in Section \ref{AlgebraicDefinition1} for the algebraic definition of gluing data. 
\end{rem}

\begin{rem}\label{RemarkOrientation}
Our notation for oriented edges of $\Sigma$ will be slightly different from the one used in \cite{ABBR1}. An oriented edge of $\Sigma$ will be denoted by $\overrightarrow{e}=xy$, where $x$ and $y$ are the endpoints and $\overrightarrow{e}$ starts at $x$ and ends at $y$. The reversed edge will be denoted by $\overleftarrow{e}=yx$.  
\end{rem}

\section{A simultaneous semistable reduction theorem for tame coverings}\label{SectionCoverings}

In this section, we prove a simultaneous semistable reduction theorem for residually tame coverings of an algebraic curve $X$. 
The theorem will be used to construct a functor from the category of residually tame coverings of $(X,D)$ to the category of 
 tame coverings of a metrized compled $\Sigma$ associated to $(X,D)$. 

\vspace{0.3cm}

Let $\mathcal{X}$ be a (strongly) semistable model for $X/K$ and consider a $K$-rational point $P:\mathrm{Spec}(K)\rightarrow{X}$. By the valuative criterion of properness, this gives rise to a unique $P_{R}:\mathrm{Spec}(R)\rightarrow{\mathcal{X}}$ such that the composition of $P_{R}$ with $\mathrm{Spec}(K)\rightarrow{\mathrm{Spec}(R)}$ is $P$ (here we identify the generic fiber of $\mathcal{X}$ with $X$). Taking the composition of $P_{R}$ with the map $\mathrm{Spec}(k)\rightarrow{\mathrm{Spec}(R)}$, this gives a $k$-rational point $\overline{P}:\mathrm{Spec}(k)\rightarrow{\mathcal{X}_{s}}$. We call the map $\mathrm{red}:X(K)\rightarrow{\mathcal{X}_{s}(k)}$ that takes $P$ to $\overline{P}$ 
the reduction map. We now define a (strongly) semistable model for a marked curve $(X,D)$ to be a (strongly) semistable model for $X$ such that the restriction of the reduction map to $D$ gives an injection $D\rightarrow{\mathcal{X}^{\mathrm{sm}}_{s}(k)}$. Here $\mathcal{X}^{\mathrm{sm}}_{s}$ is the smooth part of $\mathcal{X}_{s}$. 
We can now state the simultaneous semistable reduction theorem for residually tame coverings in terms of marked semistable models as follows. 

\begin{reptheorem}{MainThm2}
Let $\phi:X'\rightarrow{X}$ be a residually tame covering of smooth proper connected algebraic curves over $K$ and let $\mathcal{X}$ be a (strongly) semistable model for $(X,D)$, where $D$ is the branch locus of $\phi$. 
Let $\mathcal{X}'$ be the normalization of $\mathcal{X}$ in the function field $K(X')$. Then $\mathcal{X}'$ is (strongly) semistable and $\mathcal{X}'\rightarrow{\mathcal{X}}$ is a finite morphism of (strongly) semistable models over $R$.
\end{reptheorem}

\begin{rem}
The parentheses around the word strongly mean that if we add the assumption that $\mathcal{X}$ be strongly semistable, then $\mathcal{X}'$ is also strongly semistable. We will also use this convention for formal strongly semistable models and strongly semistable vertex sets.  
\end{rem}

\begin{rem}\label{RemarkEquivalent3}
We will prove the Berkovich-theoretic version of Theorem \ref{MainThm2} that was given in the introduction. We quickly explain how the version in the introduction implies the one given here, we leave the other implication to the reader. 
Let $\mathcal{X}$ be a (strongly) semistable model of $(X,D)$. By taking the $\varpi$-adic completion of this model, we obtain a formal (strongly) semistable model $\mathcal{X}_{\mathrm{form}}$. This formal model in turn corresponds to a (strongly) semistable vertex set $V$ of $(X,D)$ by \cite[Theorem 5.8]{ABBR1}. We will prove that the inverse image of $V$ in $X'^{\mathrm{an}}$ is a (strongly) semistable vertex set $V'$ of $(X',D')$, where $D'=\phi^{-1}(D)$. By \cite[Theorem 5.13]{ABBR1}, there is then a finite morphism $\mathcal{X}'_{\mathrm{form}}\rightarrow{\mathcal{X}_{\mathrm{form}}}$ between the corresponding formal (strongly) semistable models and this corresponds to a finite morphism $\mathcal{X}'\rightarrow{\mathcal{X}}$ of (strongly) semistable models. We now note that a semistable model is automatically normal. Indeed, let $U=\mathrm{Spec}(A)$ be an open affine in $\mathcal{X}$. 
Since $X$ is normal, we know that $A\otimes_{R}{K}$ is normal. If $a\in{\mathrm{Frac}(A)}$ is integral over $A$, we then have $a\in{A\otimes_{R}{K}}$. We now use the natural supremum norm on $A\otimes_{R}{K}\subset{\hat{A}}\otimes_{R}{K}$, where $\hat{A}$ is the $\varpi$-completion of $A$. Since integral elements are power-bounded (see \cite[Theorem 3.1.17]{Bosch2014}), we find that $|a|\leq{1}$. But $\hat{A}$ is equal to the power-bounded elements in ${\hat{A}}\otimes{K}$ by \cite[Proposition 3.13]{BPR11} (this is where we use the reducedness), so $a\in{\hat{A}\cap{(A\otimes_{R}{K})}}=A$, as desired. 
 We thus find that $\mathcal{X}'$ is normal, which together with the finiteness over $\mathcal{X}$ implies that it is equal to the 
 normalization of $\mathcal{X}$ in $K(X')$.     

\end{rem}

To prove Theorem \ref{MainThm2}, we first recall the following important theorem by Berkovich. 
\begin{theorem}\label{BerkovichTheorem2}
{\it{(Topologically tame finite \'{e}tale Galois coverings of open disks and annuli)}}
\begin{enumerate}
\item Let $U$ be an open disk and let $\psi:V\rightarrow{U}$ be a finite topologically tame \'{e}tale Galois covering. 
Then $V$ is a disjoint union of open disks.
\item Let $U$ be an open annulus, let $V$ be connected and let $V\rightarrow{U}$ be a finite topologically tame \'{e}tale Galois covering. 
Then there exist isomorphisms ${\mathbf{S}(a)_{+}}\simeq{U}$ and $V\simeq{\mathbf{S}_{+}(a^{1/n})}$ for $a\in{K}$ such that the composed map is given by $t\mapsto{t^{n}}$. 
\end{enumerate}
\end{theorem}
\begin{proof}
This follows from \cite[Theorems 6.3.2 and 6.3.5]{Berkovich1993} by writing $U$ as a union of closed disks or closed annuli. 
\end{proof}

\begin{proof}
{\it{(Of Theorem \ref{MainThm2})}} Let $\phi:X'\to{X}$ be the morphism of curves and let $\overline{\phi}:\overline{X}\to{X}$ be its Galois closure with Galois group $G$. By basic ramification theory, we have that $D$ is again the branch locus of $\overline{\phi}$, we write $\overline{D}=\overline{\phi}^{-1}(D)$ for the ramification locus. Throughout the proof, we will denote the analytified morphisms $X'^{\mathrm{an}}\to{X^{\mathrm{an}}}$ and $\overline{X}^{\mathrm{an}}\to{X^{\mathrm{an}}}$ by $\phi$ and $\overline{\phi}$ again to ease notation. Let $U$ be a connected component of $X^{\mathrm{an}}\backslash{V}$ that is an open annulus or an open disk. 
As before the proof of Lemma \ref{GaloisClosureCriterion}, we conclude that $\overline{\phi}^{-1}(U)\rightarrow{U}$ is a $G$-torsor and thus Galois (in particular, it is \'{e}tale). 
This also implies that the covering on any connected component of $\overline{\phi}^{-1}(U)$ is Galois, with Galois group the stabilizer of the component.  

Suppose now that $\overline{\phi}^{-1}(U)\rightarrow{U}$ is topologically tame. 
We then conclude by Theorem \ref{BerkovichTheorem2} that %
$\overline{\phi}^{-1}(U)$ is a disjoint union of open annuli or open disks. On the annuli, the covering is Kummer and on the open disks the covering is trivial.  This then also implies that the coverings $\phi^{-1}(U)\rightarrow{U}$ are of the above type. If we now consider the inverse image $V'$ of the semistable vertex set $V$ of $X^{\mathrm{an}}$ in $X'^{\mathrm{an}}$, then we see that the complement $X'^{\mathrm{an}}\backslash{V'}$ is a disjoint union of open annuli and open disks. In other words, $V'$ is a semistable vertex set for $X'^{\mathrm{an}}$. The strongly semistable case also follows, since the induced morphism of metrized complexes from Proposition \cite[Corollary 4.28]{ABBR1} is finite, and thus $\Sigma$ is loopless if and only if $\Sigma'$ is loopless. 

We now show that $\overline{\phi}$ is topologically tame over $U$. We first note that $\overline{\phi}$ and $\phi$ are topologically tame over any type-$3$ point.  Indeed, this follows from the fact that for any point $x'\in{X'^{\mathrm{an}}}$ lying over $x$, we have that $\mathcal{H}(x)\subset{\mathcal{H}(x')}$ is residually tame if and only if $p\nmid{[|\mathcal{H}(x')^{*}|:|\mathcal{H}(x)^{*}|]}$. The latter in turn follows from $\tilde{\mathcal{H}}(x')= \tilde{\mathcal{H}}(x)=K$ and \cite[Proposition 2.4.7]{Berkovich1993}. We conclude using Lemma \ref{GaloisClosureCorollary} that $\overline{\phi}$ is topologically tame at any type-$3$ point. 

Suppose now for a contradiction that $\overline{\phi}$ is not topologically tame over some type-$2$ point $x\in{U}$ and let $\overline{x}$ be a point where $\overline{\phi}$ fails to be topologically tame. We will do the case where $U$ is an open annulus, the open disk case is similar. We write $x'$ for its image in $X'^{\mathrm{an}}$. By the results in \cite{ABBR1}, we can find a pair of skeleta $\Sigma'_{1}$ and $\Sigma_{1}$ of $X'^{\mathrm{an}}$ and $X^{\mathrm{an}}$ respectively such that $x'\in{\Sigma'_{1}}$, $x\in{\Sigma_{1}}$ and $\phi^{-1}(\Sigma_{1})=\Sigma'_{1}$. This in turn induces a finite morphism of triangulated marked curves $(X',V'_{1},D')\rightarrow{(X,V_{1},D)}$. From the residual tameness of $\phi$, we conclude that the morphism of triangulated marked curves is a tame covering. 

Consider the retraction morphism $\tau_{\Sigma}:X^{\mathrm{an}}\rightarrow{\Sigma}$ with $y:=\tau_{\Sigma}(x)$. Note that the set of points in $\Sigma_{1}$ that retract to $y$ under $\tau_{\Sigma}$ is a tree. For any type-$2$ point in the vertex set of $\Sigma_{1}$ that retracts to $y$ (and not equal to $y$), we claim that the covering is trivial. Indeed, we can prove this inductively by starting at the leaf-vertices of $\Sigma_{1}$ retracting to $y$. Let $z$ be such a leaf-vertex and choose a tangent direction outside $\Sigma_{1}$. Since $\phi$ is piecewise linear, we can find for any such tangent direction a geodesic $l':[a,b]\rightarrow{X'^{\mathrm{an}}}$ such that the induced map $l:[am,bm]\rightarrow{X^{\mathrm{an}}}$ is a geodesic, see \cite[Definitions 4.4 and 4.21]{ABBR1}. Here $m$ is the dilation factor.   
Suppose now that $m$ is non-trivial in a direction outside $\Sigma_{1}$. Note that $m$ is equal to the local degree $[\mathcal{H}(w'):\mathcal{H}(w')]$ for a pair of points in a geodesic representing the tangent direction,  
see \cite[Propositions 2.2.1(2) and 2.5(1)]{BPRa1}. A geodesic representing this direction contains points of type $2$ and we obtain a contradiction using \cite[Proposition 4.35]{ABBR1}. Since $z$ is a leaf-vertex, there is only one direction in which the dilation factor can be non-trivial. But then the dilation factor has to be one, since 
this would otherwise give a tame covering of $\mathbf{P}^{1}$ ramified over only one point by \cite[Theorem 4.23(2)]{ABBR1}, a contradiction. We conclude that $\phi$ is split over $z$. We then inductively continue this argument and see that $\phi$ is split over all type-$2$ points retracting to $y$ and not equal to $y$. In particular, we must have $x=y=\tau_{\Sigma}(x)$, which implies that $x$ is a type-$2$ point in an edge of $\Sigma$. 

We now find that there are only two directions in which the dilation factor can be non-trivial. Indeed, by the previous argument the tangent directions belonging to $\Sigma_{1}$ are covered, and 
the covering is split over the other tangent directions. The tame covering of residue curves $C_{x'}\rightarrow{C_{x}=\mathbf{P}^{1}}$ is then ramified over exactly two points. This implies that $C_{x'}\rightarrow{\mathbf{P}^{1}}$ is a Kummer covering, of degree equal to the ramification degree in the two directions. This ramification degree is prime to $p$, since we can detect this using type-$3$ points. We conclude that $\mathcal{H}(x)\rightarrow{\mathcal{H}(x')}$ is Galois of degree coprime to $p$. By doing this for every point $x'$ in $X'^{\mathrm{an}}$ lying over $x$, we find that all the corresponding morphisms $\mathcal{H}(x)\rightarrow{\mathcal{H}(x')}$ for $x'$ lying over $x$ are Galois of degree coprime to $p$.    Using Lemma \ref{GaloisClosureCriterion}, we conclude that the degree of $\mathcal{H}(\overline{x})$ over $\mathcal{H}(x)$ is coprime to $p$ for any $\overline{x}$ over $x$, a contradiction. We conclude that $\phi$ is topologically tame over any type-$2$ point in $U$. Since the map $\overline{\phi}$ is piecewise linear, 
we then also easily conclude that $\overline{\phi}$ is topologically tame at points of type $1$ and $4$ in $U$. 

\end{proof}

\begin{rem}
For Galois-topologically tame coverings, the material in \cite{ABBR1} on lifting morphisms to skeleta is not necessary, as we are done after the first paragraph of the proof. Since the theorems by Berkovich do not rely on any semistable reduction theorem, this gives a stand-alone proof of this theorem. The material in \cite{ABBR1} on lifting morphisms however does rely on the semistable reduction theorem, so the second part does not give a stand-alone proof of this theorem for residually tame coverings. In closing, we note that the conclusion in the proof of Theorem \ref{MainThm2} is similar to that of \cite[Lemma 3.4.2]{TEMKIN2017}. Namely, if we have a residually tame morphism together with a skeleton of that morphism, then this morphism splits outside the skeleton. 
\end{rem}

\section{Enhanced tame coverings of metrized complexes}\label{RigidifiedCoverings}

In this section we introduce the category of {{enhanced}} tame coverings of a metrized complex $\Sigma$. 
These coverings consist of a tame covering of a metrized complex, together with a set of gluing data. Adding this set of gluing data gives a unique lift to the category of algebraic coverings by the results in \cite{ABBR1}, and we show that the category of residually tame coverings is equivalent to the category of enhanced tame coverings of a metrized complex.    

\subsection{Enhanced tame coverings}\label{RigidMorphisms2}

We now introduce the notion of an {\it{enhanced covering}} of a metrized complex $\Sigma$ and the notion of a morphism of enhanced coverings. This gives a category $\mathrm{Cov}_{\mathcal{G}}(\Sigma)$ and we show in Theorem \ref{MainThm3} that this category is equivalent to the category of residually tame coverings of a marked curve. 

\begin{rem}
We assume throughout Section \ref{RigidifiedCoverings} that $X$ and $\Sigma$ are connected. Metrized complexes $\Sigma'$ considered in coverings  $\Sigma'\rightarrow{}\Sigma$ are not necessarily connected. The semistable vertex sets in this section are strongly semistable (meaning that the skeleton $\Sigma$ is loopless). 
\end{rem}

\begin{mydef}{\it{(Enhanced tame coverings)}}\label{RigidCoverings1}
An {\it{enhanced tame covering}} of a metrized complex $\Sigma$ associated to a triangulated marked curve $(X,V\cup{D})$ is a pair $\phi_{\mathfrak{g}}:=(\phi,\mathfrak{g})$ consisting of:
\begin{enumerate}
\item A tame covering $\phi:\Sigma'\rightarrow{\Sigma}$ of metrized complexes of $k$-curves,
\item An element $\mathfrak{g}$ of the set of gluing data $\mathcal{G}(\Sigma',X)$, see Section \ref{OverviewSection} and Remark \ref{IdentificationGluingData547}. 
\end{enumerate} 
We will also refer to such a pair as an {\it{enhanced covering}} of $\Sigma$. 
\end{mydef}



Let $\phi_{\mathfrak{g}}=(\phi,\mathfrak{g})$ be an enhanced covering. For every pair of vertices $x'\in{\Sigma'}$ and $\phi(x')=:x\in{\Sigma}$, there is a unique morphism of star-shaped curves $Y(x')\rightarrow{Y(x)}$ associated to it by \cite[Theorem 6.18]{ABBR1}. For every closed edge $e'$ in $\Sigma'$ that $x'$ belongs to, consider the  inclusion $\tau_{x}^{-1}(e'^{o})\subset{Y(x')}$, where $e'^{o}$ is the open edge corresponding to $e'$ and $\tau_{x'}$ is the retraction map for $Y(x')$. If $y'$ is the other endpoint of $e'$, then abstractly we have a $\tau^{-1}(e^{0})$-isomorphism $\tau_{x'}^{-1}(e'^{o})\rightarrow{\tau_{y'}^{-1}(e'^{o})}$. We recall that a set of gluing data is a set of these isomorphisms for all oriented edges $\overrightarrow{e}'={x'y'}$. We now define morphisms of enhanced tame coverings. 


\begin{mydef}{\it{(Morphisms of enhanced coverings)}}\label{MorphismRigidCoverings1}
Let $\phi_{1,\mathfrak{g}_{1}}$ and $\phi_{2,\mathfrak{g}_{2}}$ be two enhanced coverings of a metrized complex $\Sigma$. We write $\phi_{i}:\Sigma_{i}\rightarrow{\Sigma}$ for the morphisms of metrized complexes and $\Theta_{i,\overrightarrow{e}}$ for the isomorphisms arising from the gluing data.  A morphism  $\phi_{1,\mathfrak{g}_{1}}\rightarrow{\phi_{2,\mathfrak{g}_{2}}}$ is a morphism $\psi:\Sigma_{1}\rightarrow{\Sigma_{2}}$ with the following two properties:
\begin{enumerate}
\item The diagram 
\begin{center}
\begin{tikzcd}
\Sigma_{1} \arrow{r}{\psi} \arrow{d}{\phi_{1}} & \Sigma_{2} \arrow{dl}{\phi_{2}} \\
\Sigma & 
\end{tikzcd}
\end{center}
commutes. 
\item The induced diagram 
\begin{center}
\begin{tikzcd}
\tau^{-1}_{x}(e_{1}^{o}) \arrow{r}{\Theta_{1,\overrightarrow{e_{1}}}} \arrow{d}{\psi} & \tau^{-1}_{y}(e_{1}^{o}) \arrow{d}{\psi}\\
\tau^{-1}_{\psi(x)}(e_{2}^{o}) \arrow{r}{\Theta_{2,\overrightarrow{e_{2}}}} \arrow{d}{\phi_{2}} & \tau^{-1}_{\psi(y)}(e_{2}^{o}) \arrow{d}{\phi_{2}}  \\
\tau^{-1}(e^{o}) \arrow{r}{\mathrm{id}} & \tau^{-1}(e^{o}) 
\end{tikzcd}
\end{center}
commutes for every edge $e_{1}\in\Sigma_{1}$ with images $e_{2}\in\Sigma_{2}$ and $e\in\Sigma$. Here we again write $\psi$ and $\phi_{2}$ for the induced unique maps (see \cite[Theorem 6.18]{ABBR1})  on the corresponding star-shaped curves (see \cite[Definition 6.2]{ABBR1}) and the restrictions of these maps to subannuli. 
\end{enumerate}

\end{mydef}

\begin{rem}
The composition of the maps $\psi:\tau^{-1}_{x}(e^{o}_{1})\rightarrow{\tau^{-1}_{\psi(x)}(e_{2}^{o})}$ and $\phi_{2}:\tau^{-1}_{\psi(x)}(e_{2}^{o})\rightarrow{\tau^{-1}(e^{o})}$ is equal to $\phi_{1}:\tau^{-1}_{x}(e^{o}_{1})\rightarrow{\tau^{-1}(e^{o})}$. Indeed, this follows from the commutativity in the first condition of Definition \ref{MorphismRigidCoverings1}, together with the fact that the maps on the star-shaped curves are uniquely induced from the maps of metrized complexes. 
\end{rem}

\begin{rem}
We note that any morphism $\Sigma_{1}\rightarrow{\Sigma_{2}}$ of enhanced tame coverings of $\Sigma$ is automatically tame. Indeed, any subextension of a separable extension is separable (which gives tameness at the vertices) 
and ramification degrees are multiplicative in towers (which gives tameness at the edges).  
\end{rem}

\begin{rem}
We will occasionally denote the commutative diagram of metrized complexes in Definition \ref{MorphismRigidCoverings1} by $\Sigma_{1}\rightarrow{\Sigma_{2}}\rightarrow{\Sigma}$ to ease notation. 
\end{rem}

\begin{mydef}\label{CategoryRigidifiedCoverings}{\it{(Category of enhanced tame coverings)}}
Let $\Sigma$ be a metrized complex. We define the {\it{category $\mathrm{Cov}_{\mathcal{G}}(\Sigma)$ of enhanced tame coverings}} of $\Sigma$ to be the category consisting of:
\begin{enumerate}
\item Enhanced tame coverings of $\Sigma$ as its objects (see Definition \ref{RigidCoverings1}) and
\item Morphisms of enhanced coverings to be the morphisms between the objects (see Definition \ref{MorphismRigidCoverings1}).
\end{enumerate}
We will also call this the {\it{category of enhanced coverings}} of $\Sigma$. 
\end{mydef}

\begin{exa}\label{ExampleRigidifiedCoverings1}
Let us examine \cite[Example 7.8]{ABBR1} from the viewpoint of enhanced coverings. The set-up is as follows. Assume that $\mathrm{char}(k)\neq{2}$. Consider the Tate curve $E$, given by $K^{*}/\langle{q}\rangle$ for some $q\in{K^{*}}$ with $v(q)>0$ and let $\Sigma'\rightarrow{\Sigma}$ be the degree $2$ covering of metrized complexes as in \cite[Example 7.8]{ABBR1}. 
The gluing data $\mathcal{G}(\Sigma',E)$ of the covering $\Sigma'\rightarrow{\Sigma}$ is then easily seen to consist of four elements, corresponding to the choice of an automorphism per edge. The automorphism group $\mathrm{Aut}_{\Sigma}(\Sigma')$ of metrized complexes then also has order four. If we however choose a set of gluing data $\mathfrak{g}$ and consider the corresponding enhanced automorphism group $\mathrm{Aut}_{\Sigma,\mathfrak{g}}(\Sigma')$, then it has order two. 

\end{exa}

Let $\phi:(X',V'\cup{D'})\rightarrow{(X,V\cup{D})}$ be a tame covering of triangulated marked curves, see \cite[Definition 3.8]{ABBR1}. 
The inverse image $\Sigma':=(\phi^{\mathrm{an}})^{-1}(\Sigma)$ is then a skeleton of $(X',D')$ and by \cite[Corollary 4.28]{ABBR1} we obtain a natural tame covering of metrized complexes $\Sigma'\rightarrow{\Sigma}$. In terms of that paper, we say that the morphism $\phi$ is a lifting of $\Sigma'\rightarrow{\Sigma}$. 

We now associate an enhanced tame covering of $\Sigma$ to $\phi$. 
Let $x\in{V}$ and let $Y(x)$ be the canonical star-shaped neighborhood of $x$, as in \cite[Sections 6 and 7]{ABBR1}. We then consider the inverse image of $Y(x)$ in $X'^{\mathrm{an}}$. 
This inverse image is a disjoint union of star-shaped curves $Y(x')$ for $x'\in{X'^{\mathrm{an}}}$ mapping to $x$. We similarly write $y$, $Y(y)$ and $Y(y')$ for a vertex $y$ adjacent to $x$. For any two vertices $x'$ and $y'$ lying above $x$ and $y$ respectively, we consider the intersection $Y(x')\cap{Y(y')}$, which is a disjoint union of open edges: $Y(x')\cap{Y(y')}\simeq{}\coprod{}\tau'^{-1}(e'^{o})$. Here the disjoint union is over all edges $e'$ that contain $x'$ and $y'$ and $\tau'$ is the retraction map on $X'^{\mathrm{an}}$.  
Since the retraction maps $\tau_{x'}$ and $\tau_{y'}$ on the star-shaped curves $Y(x')$ and $Y(y')$ are the restrictions of the retraction map $\tau'$, we then 
have canonical $\tau^{-1}(e^{o})$-isomorphisms $\tau_{x'}^{-1}(e'^{o})\simeq\tau^{-1}(e'^{o})\simeq\tau_{y'}^{-1}(e'^{o})$, which gives a set of gluing data $\mathfrak{g}$ for $\Sigma'\rightarrow{\Sigma}$. We thus have an associated enhanced covering $\phi_{\mathfrak{g}}$ for $\phi$. Furthermore, suppose that we have a morphism of tame coverings of a triangulated marked curve $(X,V\cup{D})$, which is a commutative diagram 
\begin{center}
\begin{tikzcd}
(X_{1},V_{1}\cup{D_{1}}) \arrow{r}{\psi} \arrow{d}{\phi_{1}} & (X_{2},V_{2}\cup{D_{2}}) \arrow{dl}{\phi_{2}} \\
(X,V\cup{D}) & 
\end{tikzcd}.
\end{center}
This in turn comes from a commutative diagram of 
analytifications 
\begin{center}
\begin{tikzcd}
X^{\mathrm{an}}_{1} \arrow{r}{\psi^{\mathrm{an}}} \arrow{d}{\phi^{\mathrm{an}}_{1}} & X^{\mathrm{an}}_{2} \arrow{dl}{\phi^{\mathrm{an}}_{2}} \\
X^{\mathrm{an}} & 
\end{tikzcd}.
\end{center}
Using \cite[Corollary 4.28]{ABBR1}, we see that this induces a commutative diagram of 
metrized complexes
\begin{center}
\begin{tikzcd}
(\phi^{\mathrm{an}}_{1})^{-1}(\Sigma) \arrow{r}{\psi} \arrow{d}{\phi^{\mathrm{an}}_{1}} & (\phi^{\mathrm{an}}_{2})^{-1}(\Sigma) \arrow{dl}{\phi^{\mathrm{an}}_{2}} \\
\Sigma  & 
\end{tikzcd}.
\end{center}
Since the diagram of analytifications is commutative, we directly obtain that 
 the gluing data defined above commutes with $\psi^{\mathrm{an}}$ and the $\phi^{\mathrm{an}}_{i}$ by considering the canonical opens $Y(x_{i})$, $Y({y_{i}})$ and $Y({x_{i}})\cap{Y({y_{i}})}$ in $X^{\mathrm{an}}_{i}$. 
We thus see that we have an induced morphism of enhanced tame coverings of $\Sigma$. In other words, we have a functor $\mathcal{F}_{\mathrm{tri}}$ from the category $\mathrm{Tame}(X,V\cup{D})$ of tame coverings of a triangulated marked curve to the category $\mathrm{Cov}_{\mathcal{G}}(\Sigma)$ of enhanced tame coverings of $\Sigma$.  

 We now compare morphisms in these categories $\mathrm{Tame}(X,V\cup{D})$ and $\mathrm{Cov}_{\mathcal{G}}(\Sigma)$. To ease notation, we adopt the notation $X_{i}:=(X_{i},V_{i}\cup{D}_{i})$ and $X:=(X,V\cup{D})$ for triangulated marked curves in this lemma. 

\begin{lemma}\label{Fullness2}
Let $\phi_{i}:X_{i}\rightarrow{X}$ be tame coverings of a triangulated marked curve $X$ with enhanced tame coverings $\Sigma_{i}\rightarrow{\Sigma}$ arising from the functor $\mathcal{F}_{\mathrm{tri}}$ constructed above. 
Then
\begin{equation}
\mathrm{Hom}_{X}(X_{1},X_{2})
\simeq{\mathrm{Hom}_{\Sigma}(\Sigma_{1},\Sigma_{2})}.
\end{equation}
\end{lemma}
\begin{proof}
We first show the injectivity of the induced map. Suppose that there are two coverings $\psi_{i}$ that map to the same morphism of enhanced coverings.  
For every pair of vertices $x_{1}$ and $x_{2}$ with $\psi_{i}(x_{1})=x_{2}$, we have a unique extension of the algebraic covering $C_{x_{1}}\rightarrow{C_{x_{2}}}$ to a covering of star-shaped curves 
$Y(x_{1})\rightarrow{Y(x_{2})}$ by \cite[Theorem 6.18]{ABBR1}. But these open neighborhoods cover $X^{\mathrm{an}}_{1}$ and $X^{\mathrm{an}}_{2}$,  
so we conclude that $\psi_{1}=\psi_{2}$.

Let $\psi:\Sigma_{1}\rightarrow{\Sigma_{2}}\rightarrow{\Sigma}$ be a morphism of  enhanced tame coverings. We write $Y(x_{i})$ for the star shaped curves corresponding to the vertices $x_{1}$ and $x_{2}$, where $x_{1}$ maps to $x_{2}$. Similarly, we write $Y(y_{i})$ for the star-shaped curves corresponding to adjacent vertices. The commutativity of the diagram in the second condition of Definition \ref{MorphismRigidCoverings1} then implies that the local morphisms $Y(x_{1})\rightarrow{Y(x_{2})}\rightarrow{Y(x)}$ and $Y(y_{1})\rightarrow{Y(y_{2})}\rightarrow{Y(y)}$ lift to a well-defined global morphism on their union. But the union over all vertices is exactly 
$X^{\mathrm{an}}_{i}$, so we conclude that $\psi$ lifts to coverings $X^{\mathrm{an}}_{1}\rightarrow{X^{\mathrm{an}}_{2}}\rightarrow{X^{\mathrm{an}}}$. 

\end{proof}

\begin{cor}\label{EquivalenceTame3}
Let $\mathrm{Tame}(X,V\cup{D})$ be the category of tame coverings of a triangulated marked curve $(X,V\cup{D})$ and let $\mathrm{Cov}_{\mathcal{G}}(\Sigma)$ be the category of enhanced tame coverings. Then $\mathcal{F}_{\mathrm{tri}}$ induces an equivalence of categories.  
\end{cor}
\begin{proof}
By Lemma \ref{Fullness2}, we see that the functor $\mathcal{F}_{\mathrm{tri}}$ is fully faithful. For any enhanced covering $(\Sigma'\rightarrow{\Sigma},\mathfrak{g})$, we glue the local coverings on the star-shaped curves using $\mathfrak{g}$ to obtain a smooth proper analytic space $X'^{\mathrm{an}}$ with a covering $X'^{\mathrm{an}}\rightarrow{X^{\mathrm{an}}}$ as in \cite[Theorem 7.4]{ABBR1}. This comes from an algebraic covering $X'\rightarrow{X}$ and we easily verify that it has the correct properties.  

\end{proof}


We now consider the category $\mathrm{Cov}_{\mathrm{Tame}}(X,D)$ 
of residually tame coverings of the marked curve $(X,D)$. We fix a strongly semistable vertex set $V$ of $(X,D)$ for the remainder of the section. Using Theorem \ref{MainThm2}, we see that the inverse image $V'=(\phi^{\mathrm{an}})^{-1}(V)$ for any 
residually tame covering $\phi:X'\rightarrow{X}$ is a strongly semistable vertex set of $(X',D')$, where $D'$ is the inverse image of $D$. We moreover have the following:
\begin{lemma}\label{GTopToTriangulated}
Let $V$ be a fixed strongly semistable vertex set for $(X,D)$ and let $(X',D')\rightarrow{(X,D)}$ be a residually tame \'{e}tale covering. Let $\phi_{\mathrm{tri}}:(X',V'\cup{D'})\rightarrow{(X,V\cup{D})}$ be the finite morphism of triangulated marked curves induced from Theorem \ref{MainThm2}. Then $\phi_{\mathrm{tri}}$ is a tame covering of triangulated marked curves. This induces a functor from the category of residually tame coverings of $(X,D)$ to the category of tame coverings of $(X,V\cup{D})$.  
\end{lemma}
\begin{proof}
Let $x'$ be a type-$2$ point in $X'^{\mathrm{an}}$ with $\phi^{\mathrm{an}}(x')=x$. 
Then by \cite[Proposition 2.4.7]{Berkovich1993}, the extension of residue fields $\tilde{\mathcal{H}}(x')\supset{\tilde{\mathcal{H}}(x)}$ is separable. As in the proof of Theorem \ref{MainThm2}, the expansion factor of an edge $e$ is just the local degree $[\mathcal{H}(x'):\mathcal{H}(x)]$ for any point in $e$. We then take a point of type $3$ and conclude that the degree is not divisible by $p$. This implies that $\phi$ is a tame covering. We leave the functoriality to the reader.   


\end{proof}

We now consider the composite of the functors in 
Corollary \ref{EquivalenceTame3} 
and Lemma \ref{GTopToTriangulated}. 

\begin{mydef}{\it{(Tropicalization functor)}}
Let $\mathcal{F}_{\Sigma}:\mathrm{Cov}_{\mathrm{Tame}}(X,D)\rightarrow{\mathrm{Cov}_{\mathcal{G}}(\Sigma)}$ be the composite %
of the functors in Corollary \ref{EquivalenceTame3} and Lemma \ref{GTopToTriangulated}.  
We call this the tropicalization functor associated to $\Sigma$. 
\end{mydef}


\begin{reptheorem}{MainThm3}
Let $\mathcal{F}_{\Sigma}$ be the tropicalization functor from the category of residually tame coverings of a marked curve $(X,D)$ to the category of enhanced tame coverings of $\Sigma$. Then $\mathcal{F}_{\Sigma}$ induces an equivalence of categories
\begin{equation}
\mathrm{Cov}_{\mathrm{Tame}}(X,D)\simeq{{\mathrm{Cov}_{\mathcal{G}}(\Sigma)}}.
\end{equation}
\end{reptheorem}
\begin{proof}
Using Lemma \ref{EquivalenceTame3}, we see that we only have to show that the functor from $\mathrm{Cov}_{\mathrm{Tame}}(X,D)$ to $\mathrm{Tame}(X,V\cup{D})$ defines an equivalence. To do this, it suffices to show that any tame covering of triangulated marked curves is also residually tame. By \cite[Proposition 4.35]{ABBR1} any tame covering is an isomorphism outside the skeleton $\Sigma'$. At the vertices of $\Sigma'$, it defines a separable covering, so these give tame extensions. 
For an edge $e'$ in $\Sigma'$, the covering is piecewise linear, with dilation factor coprime to $p$. As in the proof of Theorem \ref{MainThm2}, this dilation factor is the degree $[\mathcal{H}(x'):\mathcal{H}(x)]$ for the points $x'$ in $e'$. This proves that $\phi$ is residually tame.    

\end{proof}

\subsection{An algebraic definition of enhanced coverings}\label{AlgebraicDefinition1}

We now give an algebraic definition of enhanced coverings, without any reference to Berkovich spaces. We note that this section is not essential to the rest of the paper. 

Consider a metrized complex of $k$-curves $\Sigma$. 
For every vertex $x\in{V(\Sigma)}$, the residue curve $C_{x}$ over $k$ is smooth and proper. Every adjacent unoriented edge $e=xy$ then corresponds to a pair of closed points $z_{e,x}$ and $z_{e,y}$ of $C_{x}$ and $C_{y}$ respectively. We will denote the corresponding oriented edges by $\overrightarrow{e}={xy}$ and $\overleftarrow{e}=yx$ as in Remark \ref{RemarkOrientation}. Since $C_{x}$ is smooth we have  
\begin{equation}
\hat{\mathcal{O}}_{C_{x},z_{e,x}}\simeq{k[[u]]}
\end{equation} 
and similarly for $z_{e,y}$ and $C_{y}$. In particular we find that the completed local rings $\hat{\mathcal{O}}_{C_{x},z_{e,x}}$ and $\hat{\mathcal{O}}_{C_{y},z_{e,y}}$ are isomorphic. There is no canonical isomorphism however, so we are led to the following definition. 

\begin{mydef}\label{AlgebraicGluingSets1}
{\it{(Algebraic gluing sets)}}
Let $\overrightarrow{E}_{f}(\Sigma)$ be the set of finite oriented edges of $\Sigma$. An algebraic set of gluing data for $\Sigma$ is a set of isomorphisms 
\begin{equation}
\psi_{\Sigma,\overrightarrow{e}}:\hat{\mathcal{O}}_{C_{x},z_{e,x}} \rightarrow{} \hat{\mathcal{O}}_{C_{y},z_{e,y}}.
\end{equation}
Here $\overrightarrow{e}$ ranges over the set of oriented edges $\overrightarrow{e}=xy$ of $\Sigma$. 
For the same edge but with the opposite orientation, we impose the condition $\psi_{\Sigma,\overleftarrow{e}}=\psi^{-1}_{\Sigma,\overrightarrow{e}}$.  A pair $(\Sigma,\psi_{\Sigma,\overrightarrow{e}})$ is called an {\it{algebraically glued metrized complex}}. If $\Sigma$ is clear from context, then we denote $\psi_{\Sigma,\overrightarrow{e}}$ by $\psi_{\overrightarrow{e}}$.
\end{mydef}

\begin{mydef}\label{AlgEnhancedCoverings547}
{\it{(Algebraically enhanced coverings)}}
Let $(\Sigma',\psi_{\overrightarrow{e}'})$ and $(\Sigma,\psi_{\overrightarrow{e}})$ be two algebraically glued metrized complexes. An enhanced covering $\phi:(\Sigma',\psi_{\overrightarrow{e}'})\to(\Sigma,\psi_{\overrightarrow{e}})$ is a tame covering of metrized complexes $\Sigma'\rightarrow{\Sigma}$ such that for every oriented edge $\overrightarrow{e}'={x'y'}$ mapping to the oriented edge $\overrightarrow{e}={xy}$ the following diagram commutes 
 \begin{center}
\begin{tikzcd}
\hat{\mathcal{O}}_{C_{x'},z_{e',x'}}\arrow{r} & \hat{\mathcal{O}}_{C_{y'},z_{e',y'}} \\
\hat{\mathcal{O}}_{C_{x},z_{e,x}} \arrow{u} \arrow{r} & \hat{\mathcal{O}}_{C_{y},z_{e,y}} \arrow{u} 
\end{tikzcd}.
\end{center}
Here the vertical maps $\hat{\mathcal{O}}_{C_{x},z_{e,x}}\rightarrow{\hat{\mathcal{O}}_{C_{x'},z_{e',x'}}}$ and $\hat{\mathcal{O}}_{C_{y},z_{e,y}}\rightarrow{\hat{\mathcal{O}}_{C_{y'},z_{e',y'}}}$ are induced from the morphism of metrized complexes, and the horizontal maps are induced by $\psi_{\overrightarrow{e}}$ and $\psi_{\overrightarrow{e}'}$. The category of algebraically enhanced coverings of $(\Sigma,\psi_{\overrightarrow{e}})$ has as its objects the algebraically enhanced coverings of $(\Sigma,\psi_{\overrightarrow{e}})$. Morphisms between algebraically enhanced coverings are defined as in Definition \ref{MorphismRigidCoverings1}. For a given tame covering of metrized complexes $\Sigma'\rightarrow{\Sigma}$ and a fixed algebraic gluing set $\psi_{\overrightarrow{e}}$ for $\Sigma$, the set of algebraic gluing data for the pair ($\Sigma'\rightarrow{\Sigma}$, $\psi_{\overrightarrow{e}}$) is the set of all gluing data $\psi_{\overrightarrow{e}'}$ for $\Sigma'$ as in Definition \ref{AlgebraicGluingSets1} that induce an algebraically enhanced covering of $(\Sigma,\psi_{\overrightarrow{e}})$. 
 This set is denoted by $\mathcal{G}_{a}(\Sigma',\psi_{\overrightarrow{e}})$.  
\end{mydef}

\begin{rem}
If the above diagram commutes for a given pair of oriented edges $\overrightarrow{e}'$ and $\overrightarrow{e}$, then the corresponding diagram for $\overleftarrow{e}'$ and $\overleftarrow{e}$ also automatically commutes by the condition on the inverses in Definition \ref{AlgebraicGluingSets1}. 
\end{rem}

We now argue that the set of algebraic gluing data $\mathcal{G}_{a}(\Sigma',\psi_{\overrightarrow{e}})$ for a given morphism $\Sigma'\rightarrow{\Sigma}$ of metrized complexes is finite. In fact, consider the finite set $\prod_{e'\in{E(\Sigma')}}\mathbf{Z}/d_{e'/e}(\phi)\mathbf{Z}$. We can non-canonically identify $\mathcal{G}_{a}(\Sigma',\psi_{\overrightarrow{e}})$ with this set by noting that any two isomorphisms for a given oriented edge $\overrightarrow{e}'$ are related by an automorphism of $\hat{\mathcal{O}}_{C_{x'},z_{e,x'}}$ over $\hat{\mathcal{O}}_{C_{x},z_{e,x}}$. This automorphism group is just $\mathbf{Z}/d_{e'/e}(\phi)\mathbf{Z}$, so we only have to note that these isomorphisms in fact exist. This follows from the fact that there is exactly one tamely ramified extension of $\hat{\mathcal{O}}_{C_{x},z_{e,x}}\simeq{k[[u]]}$ of any given degree coprime to $\mathrm{char}(k)$.

Suppose now that we are given a triangulated marked curve $(X,V\cup{D})$ with metrized complex $\Sigma$. By Remark \ref{IdentificationGluingData547}, the set of analytic gluing data for a tame covering of metrized complexes $\Sigma'\to\Sigma$ can be identified with $\prod_{e'\in{E(\Sigma')}}\mathbf{Z}/d_{e'/e}(\phi)\mathbf{Z}$. By the considerations in the previous paragraph, the same holds for the algebraic gluing data for any fixed 
 gluing set $\psi_{\overrightarrow{e}}$ for $\Sigma$. We thus see that the analytic gluing data can be identified with the algebraic gluing data.  
Moreover, the way that morphisms were defined 
in Definitions \ref{MorphismRigidCoverings1} and 
 \ref{AlgEnhancedCoverings547} directly implies that we obtain an {\it{equivalence of categories}}. The category of enhanced coverings can thus be defined without any reference to Berkovich spaces. 

\begin{exa}\label{GenusTwoCoverings}
We determine all connected degree two enhanced coverings of the metrized complex   in Figure \ref{PlaatjeGenus2}, where $\mathrm{char}(k)\neq{2}$ and the vertex of genus one corresponds to some elliptic curve $\overline{E}/k$. Note that the genus of the complex is $2$, so a degree-$2$ covering has genus $3$ by the Riemann-Hurwitz formula. The dilation factor $d_{e'/e}(\phi)$ over the edges is either $1$ or $2$. If one of them is $2$, then there is graph-theoretically only one option, see Figure \ref{PlaatjeGenus2}. On the vertex of genus $1$, this induces an \'{e}tale degree two covering of $\overline{E}\backslash{\{p_{1},p_{2}\}}$. Using Grothendieck's results on the tame fundamental group of a curve over a field  {\normalfont{\cite[Corollaire 2.12]{SGA1}}}, one then sees that $4$ of these coverings $C\rightarrow{\overline{E}\backslash{\{p_{1},p_{2}\}}}$ are ramified over at least one $p_{i}$.
Each of these maps is ramified over the edges and the corresponding gluing data can be identified with 
$\mathbf{Z}/2\mathbf{Z}\times\mathbf{Z}/2\mathbf{Z}$. These give rise to two distinct isomorphism classes of enhanced coverings, as in the case of an elliptic curve with multiplicative reduction. 
For each of the four coverings above we thus have two liftings, giving $8$ coverings in total that are ramified over at least one edge. 


\begin{figure}[h]
\centering
\includegraphics[width=9cm]{{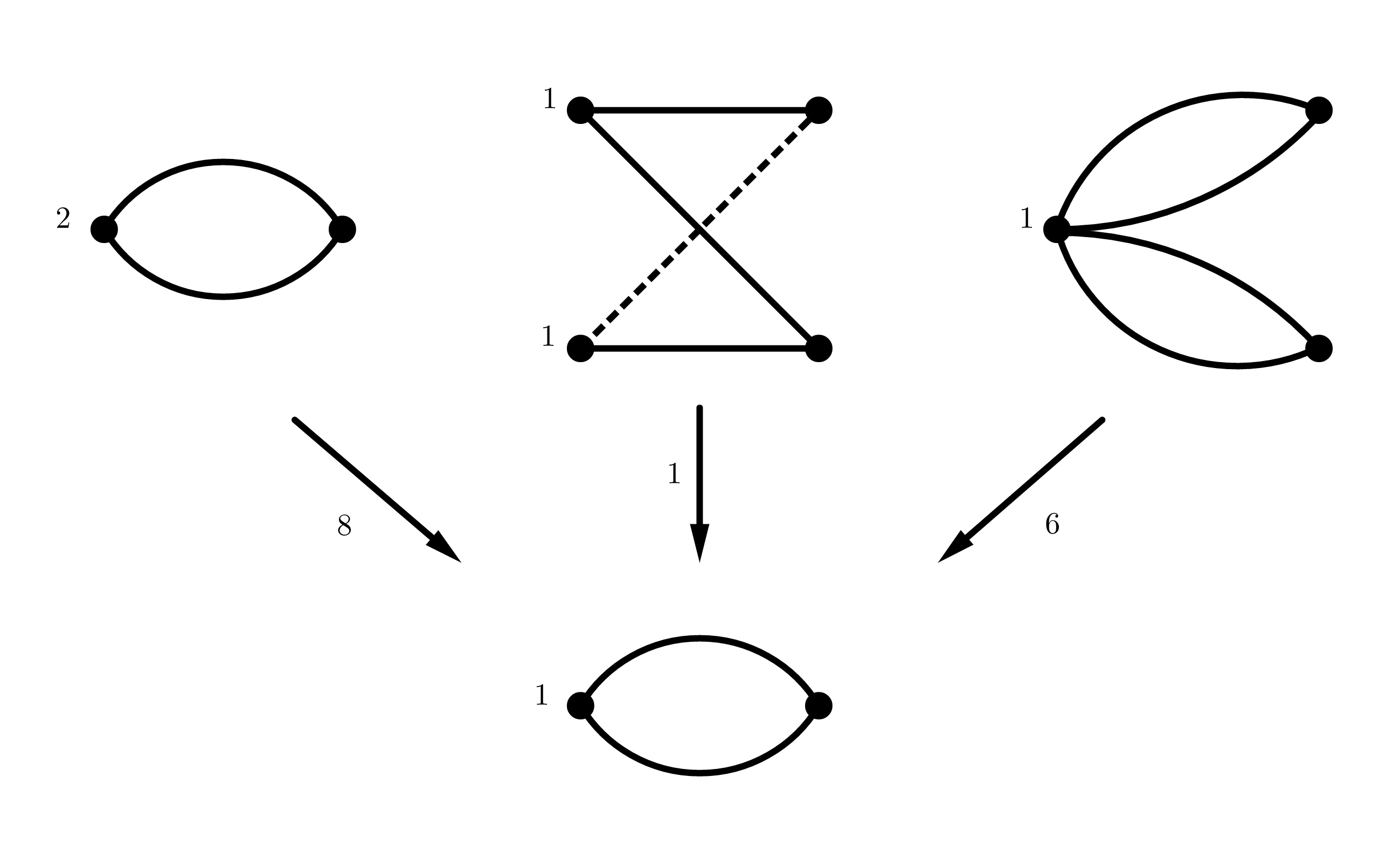}}
\caption{\label{PlaatjeGenus2} The enhanced connected coverings of degree two of a metrized complex of genus $2$. 
The integers next to the arrows represent the number of enhanced coverings corresponding to the particular graph-theoretical type. The integers next to the vertices give the genus of the vertex.  }
\end{figure}

If the dilation factor is $1$ on both edges, then there are graph-theoretically two options for a degree two covering. It is either split over the vertex of genus one or it is not. If it is not split, then it is defined by a non-trivial \'{e}tale covering of degree two of $\overline{E}$ and there are exactly $3$ such coverings. We still have some freedom in identifying the closed points corresponding to the edges however. 
This gives a total of four options per covering of $\overline{E}$. Using the non-trivial automorphism of $\overline{E}'$ over $\overline{E}$, we find that out of the four options, only two are non-isomorphic. In total, we thus obtain $6$ of these coverings. Lastly, suppose that the covering is completely split over both the edges and the vertices. There is then only one option, namely the graph-theoretical covering of degree two of the circle. In total, we find $15$ different non-trivial coverings of degree two. 


\end{exa}

\section{Fundamental groups for metrized complexes}
\label{TropicalFundamental}

In this section, we use the equivalence of categories in Theorem \ref{MainThm3} and endow the category of residually tame coverings with the structure of a Galois category, giving a natural notion of a fundamental group $\pi(\Sigma)$ for a metrized complex $\Sigma$. 
We furthermore define the notions of unramified coverings and completely split coverings above a subcomplex $\Sigma^{0}\subset{\Sigma}$. These correspond to closed subgroups of $\pi(\Sigma)$, which we call the absolute inertia group $\mathfrak{I}(\Sigma^{0})$ and decomposition group $\mathfrak{D}(\Sigma^{0})$ of $\Sigma^{0}$. 
For unmarked curves, we prove in Theorem \ref{FundamentalGroupGraph1} that the quotient $\pi_{\mathfrak{D}}(\Sigma(X)):=\pi(\Sigma)/\mathfrak{D}(\Sigma)$  is isomorphic to the profinite completion of the ordinary fundamental group of the underlying graph of $\Sigma$. Furthermore, we show that the coverings that come from the abelianizations of the quotients $\pi_{\mathfrak{D}}(\Sigma)$ and $\pi_{\mathfrak{I}}(\Sigma)$ correspond to the extensions that come from the toric and connected parts in the analytic Jacobian of the curve, see Theorem \ref{TorsionUnramifiedCoverings1}. 

\subsection{Fundamental, inertia and decomposition groups}\label{FundGroupSection}

As in Section \ref{RigidifiedCoverings}, we fix a semistable vertex $V$ for a marked curve $(X,D)$ with skeleton $\Sigma:=\Sigma(X,V\cup{}D)$. Here we again assume that $X$ is connected.  
By Theorem \ref{MainThm3}, the tropicalization functor $\mathcal{F}_{\Sigma}$ gives an equivalence of categories 
\begin{equation}
\mathcal{F}_{\Sigma}:\mathrm{Cov}_{\mathrm{Tame}}(X,D)\rightarrow{{\mathrm{Cov}_{\mathcal{G}}(\Sigma)}}.
\end{equation}

Since $\mathrm{Cov}_{\mathrm{Tame}}(X,D)$ is a Galois category, we obtain an induced Galois category structure on ${\mathrm{Cov}_{\mathcal{G}}(\Sigma)}$. 
\begin{mydef}\label{GaloisCategoryStructure1}
{\it{[Fundamental group of a metrized complex]}} Let ${\mathrm{Cov}_{\mathcal{G}}(\Sigma)}$ be the category of enhanced tame coverings of $\Sigma$. We endow it with the Galois category structure induced by Theorem \ref{MainThm3}. 
The corresponding fundamental group is denoted by $\pi(\Sigma)=\pi_{\mathrm{Tame}}(X,D)$. 
We call this group the (tame) profinite tropical fundamental group of $\Sigma$, or just fundamental group of $\Sigma$.  
\end{mydef}


Let us state some facts about the interaction between the profinite group $\pi(\Sigma)$ and the category $\mathrm{Cov}_{\mathcal{G}}(\Sigma)$. All of these results directly follow from the theory of Galois categories. 

\begin{prop}\label{Facts1}
Let $\Sigma$ be a fixed metrized complex associated to the marked curve $(X,D)$ and consider the category of enhanced tame coverings of $\Sigma$.  Then the following are true.
\begin{enumerate}
\item There is an equivalence of categories
\begin{equation}
{\mathrm{Cov}_{\mathcal{G}}(\Sigma)}\rightarrow{(\mathrm{Finite}\,\pi(\Sigma)\mhyphen\mathrm{Sets})}.
\end{equation}
\item The closed subgroups of finite index in $\pi(\Sigma)$ correspond bijectively to connected enhanced tame coverings $\Sigma'\rightarrow{\Sigma}$.
\end{enumerate}
\end{prop}

\begin{exa}\label{EllipticCurves}
Consider an elliptic curve $E/K$ with $\mathrm{char}(K)=0$. The ordinary profinite fundamental group of $E$ is $\hat{\mathbf{Z}}^{2}$. For $\mathrm{char}(k)=0$, we have that every such covering is automatically residually tame, so $\pi(\Sigma)\simeq{\hat{\mathbf{Z}}^{2}}$. If $\mathrm{char}(k)=p$, then the group depends on the reduction type of $E$. If $E$ has good reduction, then the minimal skeleton consists of a single vertex of genus $1$ and we write $\Sigma$ for this skeleton. We then have $\pi(\Sigma)\simeq{\pi(\overline{E})}$, where $\pi(\overline{E})$ is the \'{e}tale fundamental group of $\overline{E}$. Suppose that the reduced curve $\overline{E}$ has supersingular reduction. 
We then have $\pi(\Sigma)\simeq{\hat{\mathbf{Z}}'^{2}}$, where the prime means that the inverse limit in the definition of $\hat{\mathbf{Z}}$ runs over all subgroups of index coprime to $p$. The \'{e}tale coverings of degree $p^{n}$ are thus not residually tame. 
If the reduced curve $\overline{E}$ has ordinary reduction, then there are coverings of degree $p^{n}$ that are residually tame, see Example \ref{SeparableExample}. We have $\pi(\Sigma)\simeq{\hat{\mathbf{Z}}'^{2}}\times{\mathbf{Z}_{p}}$.     
Suppose now that $\Sigma$ has Betti number one and let $E/K$ be an elliptic curve with skeleton $\Sigma$. Then $\pi(\Sigma)\simeq{\hat{\mathbf{Z}}'^{2}}\times{\mathbf{Z}_{p}}$. Indeed, using the Tate uniformization 
$E^{\mathrm{an}}\simeq{\mathbf{G}^{\mathrm{an}}_{m,K}/\langle{q}\rangle}$, one easily obtains a description of all finite \'{e}tale coverings. The induced coverings on the skeleton are either topological coverings or coverings with some expansion factor $d$ on the cycle. The topological coverings are residually tame (see Proposition \ref{CategoryToric}), but the coverings with dilation factor $p^{n}$ on both edges are not. 
\end{exa}

\begin{rem}{\it{(Function fields and algebraic coverings)}}\label{FunctionFields} 
Every covering $X'\rightarrow{X}$ of connected smooth curves comes with a natural extension of function fields $K(X)\rightarrow{K(X')}$ and this is a bijective correspondence. The closed subgroups of $\pi_{\mathrm{Tame}}(X,D)=\pi(\Sigma)$ then correspond to (possibly infinite) field extensions, which are composites of finite function field extensions $K(X')\supset{K(X)}$ where $X'\rightarrow{X}$ is residually tame.     
We will create closed subgroups of $\pi(\Sigma)$ using this correspondence.   


\end{rem}

\begin{mydef}{\it{(Subcomplexes)}}
A (strict) finite subcomplex $\Sigma^{0}$ of $\Sigma$ consists of subsets $V(\Sigma^{0})\subset{V_{f}(\Sigma)}$ and $E(\Sigma^{0})\subset{E_{f}(\Sigma)}$. Here $V_{f}(\Sigma)$ and $E_{f}(\Sigma)$ denote the sets of finite vertices and edges of $\Sigma$ respectively. 
\end{mydef}

\begin{mydef}\label{UnramifiedTotallySplit}
{\it{(Unramified and completely split coverings)}}
Let $\phi:\Sigma'\rightarrow{\Sigma}$ be a connected enhanced covering of metrized complexes and let $\Sigma^{0}\subseteq{\Sigma}$ be a subcomplex. We say that
\begin{enumerate}
\item $\phi$ is {\it{metrically unramified}} (or {\it{unramified}}) above $\Sigma^{0}$ if for every edge $e'\in{E(\Sigma')}$ mapping to $e\in{E(\Sigma^{0})}$, we have $d_{e'/e}(\phi)=1$,
\item $\phi$ is {\it{completely split}} above $\Sigma^{0}$ if $\phi$ is unramified above $\Sigma^{0}$ and for every vertex $v\in{V(\Sigma^{0})}$, we have that there are $\mathrm{deg}(\phi)$ vertices $v'\in{V(\Sigma')}$ such that $\phi(v')=v$. 
\end{enumerate}
We then say that an enhanced covering of metrized complexes has any of the above properties if the connected components have these properties. The enhanced covering $\Sigma'\rightarrow{\Sigma}$ is unramified (resp. completely split) if it is unramified (resp. completely split) over the maximal finite subcomplex $\Sigma^{0}$ in $\Sigma$. We say that a residually tame covering $\phi:(X',D')\rightarrow{(X,D)}$ is unramified (resp. completely split) above a subcomplex $\Sigma^{0}\subseteq{\Sigma}$ if $\mathcal{F}_{\Sigma}(\phi)$ is unramified (resp. completely split) above $\Sigma^{0}$ (and similarly for the maximal finite subcomplex). 
\end{mydef}

\begin{rem}
Our definition of $\phi$ being completely split above an edge or vertex corresponds to $\phi$ being a topological covering above that edge or vertex. The definition of being unramified is not related to the Berkovich definition of being unramified or \'{e}tale. Indeed, any morphism $X'^{\mathrm{an}}\rightarrow{X^{\mathrm{an}}}$ will automatically be \'{e}tale outside the branch points. 
However, the morphism $\mathcal{X}'\rightarrow{\mathcal{X}}$ of strongly semistable models corresponding to $\phi:\Sigma'\rightarrow\Sigma$ is not \'{e}tale 
at the closed points corresponding to edges if and only if $d_{e'/e}(\phi)\neq{1}$. This can be seen by considering the corresponding morphism of completed rings for the ordinary double points. It is for this reason that we call these morphisms unramified at an edge if $d_{e'/e}(\phi)=1$.   

\end{rem}

\begin{lemma}\label{UnrSplitGaloisComposite}
The notions of being unramified and completely split above a subcomplex $\Sigma^{0}$ are stable under  taking composites. 
\end{lemma}
\begin{proof}
We first show that these are stable under taking Galois closures. For the notion of being unramified, we can take a type-$3$ point $x$ in an edge of $\Sigma^{0}$. If $[\mathcal{H}(x'):\mathcal{H}(x)]=1$ for every point $x'$ lying above $x$, then by Proposition \ref{GaloisClosureCriterion} we conclude that $[\mathcal{H}(\overline{x}):\mathcal{H}(x)]=1$. We can also apply the above reasoning to a point of type $2$ that is completely split to obtain the desired statement for completely split morphisms. For the composite $K(\overline{X})$ of two fields $K(X_{i})$, we argue as in \ref{CompositeTame}: we reduce to the case of two Galois extensions and then find that the composite of $\mathcal{H}(x_{1})$ and $\mathcal{H}(x_{2})$ for two points $x_{i}\in{X^{\mathrm{an}}_{i}}$ lying above $x$ is $\mathcal{H}(\overline{x})$, where $\overline{x}\in{\overline{X}^{\mathrm{an}}}$ lies above the $x_{i}$. If $[\mathcal{H}(x_{i}):\mathcal{H}(x)]=1$ for both $i$, then $[\mathcal{H}(\overline{x}):\mathcal{H}(x)]=1$, which quickly gives the desired statement for both unramified and completely split morphisms. 
\end{proof}

We now consider the closed subgroups in $\pi(\Sigma)$ corresponding to the coverings that are unramified and completely split. 

\begin{mydef}
We define $\mathfrak{I}(\Sigma^{0})$ and $\mathfrak{D}(\Sigma^{0})$ 
to be the closed subgroups of $\pi(\Sigma)$ corresponding to the coverings that are unramified and completely split above $\Sigma^{0}$ respectively. We refer to them as the (absolute) inertia and decomposition groups of $\Sigma^{0}$ respectively. If $\Sigma^{0}$ consists of all finite edges and vertices of $\Sigma$, then we denote these groups by $\mathfrak{I}(\Sigma)$ and $\mathfrak{D}(\Sigma)$.  
\end{mydef}


\begin{prop}\label{Normality1}
The subgroups $\mathfrak{I}(\Sigma^{0})$ and $\mathfrak{D}(\Sigma^{0})$ are normal subgroups of $\pi(\Sigma)$. 
\end{prop}
\begin{proof}
Let $x\in{M}$, where $M$ is the field corresponding to either $\mathfrak{I}(\Sigma^{0})$ or $\mathfrak{D}(\Sigma^{0})$ and write $L=K(X)(x)$ for the field generated by $x$ over $K(X)$. By Lemma, \ref{UnrSplitGaloisComposite}, the covering $X'\rightarrow{X}$ corresponding to $L$ is unramified or completely split above $\Sigma^{0}$, and the Galois closure then also has the same property. Thus the conjugates of $x$ are contained in $M$, which proves that it is Galois. 

\end{proof}

\begin{mydef}\label{ConnectedToricParts1}{\it{(Decomposition and inertia quotients of a subcomplex)}}
Let $\Sigma^{0}\subseteq{\Sigma}$ be a subcomplex and let 
$\mathfrak{I}(\Sigma^{0})$ and $\mathfrak{D}(\Sigma^{0})$ be the absolute inertia and decomposition groups of $\Sigma^{0}$. 
We define $\pi_{\mathfrak{I}}(\Sigma^{0}):=\pi(\Sigma)/\mathfrak{I}(\Sigma^{0})$ and $\pi_{\mathfrak{D}}(\Sigma^{0}):=\pi(\Sigma)/\mathfrak{D}(\Sigma^{0})$. If $\Sigma^{0}$ consists of all finite vertices and edges of $\Sigma$, then we write $\pi_{\mathfrak{I}}(\Sigma)$ and $\pi_{\mathfrak{D}}(\Sigma)$ for these groups.   

\end{mydef}

\begin{rem}
We will see in Section \ref{Abelianization1} that there is a natural connection between the cyclic abelian extensions coming from $\pi_{\mathfrak{D}}(\Sigma)$ and $\pi_{\mathfrak{I}}(\Sigma)$ and the cyclic  abelian extensions coming from the toric and connected parts of the analytic Jacobian of $X$. It is in this sense that we think of the groups $\pi_{\mathfrak{D}}(\Sigma)$ and $\pi_{\mathfrak{I}}(\Sigma)$ as natural {\it{non-abelian}} generalizations of the extensions coming from the toric and connected parts of the Jacobian in the tame case. 
\end{rem}
We now connect the group $\pi_{\mathfrak{D}}(\Sigma)$ to the profinite completion of the ordinary fundamental group of the graph underlying $\Sigma$. 
Let $\Gamma$ be the finite connected graph underlying $\Sigma$.  We then denote the category of finite coverings by $\mathrm{Cov}(\Gamma)$,  its profinite fundamental group by $\hat{\pi}(\Gamma)$ and its ordinary fundamental group by $\pi(\Gamma)$ (this is the only time we will use this notation for a nonprofinite group). 
The normal subgroup $\mathfrak{D}(\Sigma)$ gives rise to a Galois subcategory of $\mathrm{Cov}(\Sigma)$, which we denote by $\mathrm{Cov}_{\mathfrak{D}}(\Sigma)$. The corresponding profinite fundamental group is $\pi_{\mathfrak{D}}(\Sigma)$. 
We now have the following
\begin{prop}\label{CategoryToric}
Let $\mathrm{Cov}_{\mathfrak{D}}(X)$ be as above and consider the 
 the forgetful functor $\mathrm{Cov}_{\mathfrak{D}}(\Sigma)\rightarrow{\mathrm{Cov}(\Gamma)}$. This induces an equivalence of categories
\begin{equation}
\mathrm{Cov}_{\mathfrak{D}}(\Sigma)\rightarrow{\mathrm{Cov}(\Gamma)}.
\end{equation} 
\end{prop}
\begin{proof}
Let $\Gamma'\rightarrow{\Gamma}$ be a finite covering of finite graphs. By assigning the same length function on $\Gamma'$ as on $\Gamma$ (induced by $\Sigma$) and by assigning to every vertex of $\Gamma'$ a projective line with the right identifications of the edges, we easily obtain a finite tame covering of metrized complexes $\Sigma'\rightarrow{\Sigma}$. There is no non-trivial gluing data (indeed, $\mathrm{Aut}_{\tau^{-1}_{x}(e^{0})}(\tau_{x'}^{-1}(e'^{0}))=(1)$ for every finite vertex $x'$ with image $x$ and adjacent edges $e'$ and $e$), so this also gives a canonical enhanced morphism. This shows that the functor is essentially surjective. For any two completely split coverings $\Sigma_{i}\rightarrow{\Sigma}$, giving an enhanced $\Sigma$-covering $\Sigma_{1}\rightarrow{\Sigma_{2}}$ of metrized complexes is no different from giving a covering of graphs (since there are no non-trivial coverings of algebraic curves and no non-trivial dilation factors). This gives an isomorphism of $\mathrm{Hom}$-sets and we conclude that the categories are equivalent.  

\end{proof}

\begin{reptheorem}{FundamentalGroupGraph1}
Let $\mathfrak{D}(\Sigma)$ be the decomposition group of $\Sigma$ in $\pi(\Sigma)$. Then $\pi_{\mathfrak{D}}(\Sigma):=\pi(\Sigma)/\mathfrak{D}(\Sigma)$ is isomorphic to the profinite completion of the ordinary fundamental group of the underlying graph $\Gamma$ of a metrized complex $\Sigma$ corresponding to $X$.
\end{reptheorem}
\begin{proof}
 First, the category $\mathrm{Cov}(\Gamma)$ is a Galois category with profinite fundamental group equal to $\hat{\pi}(\Gamma)$. Indeed by \cite[Chapitre IX, num\'{e}ro 6; 19, Chapter V]{Godbillon1} or a modification of \cite[Theorem 1.38]{Hatcher1}, we find that the category of $\pi(\Gamma)$-sets is equivalent to the category of coverings of $\Gamma$ and thus the profinite completion $\hat{\pi}(\Gamma)$ is equivalent to the category of finite coverings of $\Gamma$. By Proposition \ref{CategoryToric}, this last category is equivalent to the category of finite $\pi_{\mathfrak{D}}(\Sigma)$-sets, which immediately yields the desired isomorphism by \cite[\href{http://stacks.math.columbia.edu/tag/0BN5}{Tag 0BN5}]{stacks-project}.
\end{proof}

\subsection{Filtrations of the abelianization of $\pi(\Sigma)$}\label{Abelianization1}

In this section, we study the abelianizations of the profinite groups $\pi(\Sigma)$, $\pi_{\mathfrak{I}}(\Sigma)$ and $\pi_{\mathfrak{D}}(\Sigma)$ introduced in Section \ref{FundGroupSection}. 
We start by reviewing some material on divisors on metric graphs and skeleta of Jacobians. For more details, we refer the reader to \cite{Baker2014}.

Let $X$ be a connected curve over $K$ as in Section \ref{AlgebraicAnalyticPrelim} and let $J:=J(X)$ be its Jacobian. 
We write $\mathcal{X}$ 
for a fixed strongly semistable model of $X$ with skeleton $\Sigma$ and retraction map $\tau:X^{\mathrm{an}}\rightarrow{\Sigma}$.  
By linearly extending $\tau$, we then obtain a map 
\begin{equation}
\tau_{*}:\mathrm{Div}(X)\rightarrow{}
{\mathrm{Div}_{\Lambda}(\Sigma)}.
\end{equation}
On divisors of degree zero, this leads to the following commutative diagram with exact rows: 
\begin{center}
\begin{tikzcd}
{0} \arrow[r] & {\mathrm{Prin}(X)} \arrow[r] \arrow[d] & {\mathrm{Div}^{0}(X)} \arrow[r] \arrow[d] & {J(K)} \arrow[r] \arrow[d] & {0} \\
{0} \arrow[r] & {\mathrm{Prin}_{\Lambda}(\Sigma)} \arrow[r]           & {\mathrm{Div}^{0}_{\Lambda}(\Sigma)} \arrow[r]           & {\mathrm{Jac}_{\Lambda}(\Sigma)} \arrow[r]           & {0}
\end{tikzcd}
\end{center}
Here $\mathrm{Jac}_{\Lambda}(\Sigma)$ is the $\Lambda$-valued tropical Jacobian of $\Sigma$. We will view this as a subset of the real tropical Jacobian $\mathrm{Jac}(\Sigma)=\mathrm{Jac}_{\mathbf{R}}(\Sigma)$.     
By the results in \cite{Baker2014}, this real tropical Jacobian arises naturally as the skeleton of $J^{\mathrm{an}}$. Furthermore, under this identification the retraction map 
\begin{equation}
\overline{\tau}:J^{\mathrm{an}}\rightarrow{\mathrm{Jac}(\Sigma)}
\end{equation}
 is given on $K$-valued points by the earlier maps on divisor groups. 
We now consider the kernel of $\overline{\tau}$, which is the unique compact analytic domain $J^{0}\subset{J^{\mathrm{an}}}$ that is also a formal $K$-analytic subgroup, see \cite[Corollary 6.8]{Baker2014}. 
We then have an exact sequence 
\begin{equation}
(1)\rightarrow{T^{0}}\rightarrow{J^{0}}\rightarrow{B^{\mathrm{an}}}\rightarrow{(1)}
\end{equation}
where $T^{0}\subset{T^{\mathrm{an}}}$ is an affinoid torus 
and $B$ is an abelian variety with good reduction. Note that if $T^{\mathrm{an}}=(\mathbf{G}^{\mathrm{an}}_{m,K})^{n}$, then the canonical reduction $\overline{T}$ of $T^{0}$ is isomorphic to $(\mathbf{G}_{m,k})^{n}$.
 As in \cite[Theorem 5.1.c]{Bosch1984} and \cite[Section 7.2]{Baker2014}, the reduction of $J^{0}$ is equal to the Jacobian of the special fiber $\mathcal{X}_{s}$, which gives the exact sequence 
\begin{equation}\label{ExactSequenceTorsion}
(1)\rightarrow{\overline{T}}\rightarrow{\overline{J^{0}}}\rightarrow{\prod_{i=1}^{n}\mathrm{Jac}(\Gamma_{i})=\overline{B}}\rightarrow{(1)}.
\end{equation}
Here the $\Gamma_{i}$ are the irreducible components of $\mathcal{X}_{s}$ and $\overline{T}\simeq(k^{*})^{t}$. 
This toric rank $t$ is equal to the first Betti number of the intersection graph of $\mathcal{X}_{s}$ by \cite[Chapter 7, Lemma 5.18]{liu2}. We write $a=\sum_{i=1}^{n} g(C_{i})$ for the {{abelian rank}} of $J^{0}$ and $\pi$ for the map $\overline{J}^{0}\rightarrow{{\prod_{i=1}^{n}\mathrm{Jac}(\Gamma_{i})}}=\overline{B}$ in Equation \ref{ExactSequenceTorsion}. 

The group scheme $J[n]$ is \'{e}tale over $K$ by our assumptions on $n$, so we can identify it with its $K$-rational points. These $K$-rational points give a set of type-$1$ points of $J^{\mathrm{an}}$ and we again denote these by $J[n]$. We then define 
\begin{eqnarray*}
J^{0}[n]&=&\{P\in{J[n]}:\overline{\tau}(P)=0\},\\
T[n]&=&\{P\in{J^{0}[n]}:\pi(\overline{P})=0\}.
\end{eqnarray*}
Here $\overline{P}$ is the image of $P$ in $\overline{J}^{0}$ under the reduction map. 



\begin{prop}\label{TorsionProposition1}
Let $n$ be any integer that is coprime to the residue characteristic of $K$. Then
\begin{align}
J^{0}[n]& \simeq{}{(\mathbf{Z}/n\mathbf{Z})^{t+2a}},\\
T[n]&\simeq{}{(\mathbf{Z}/n\mathbf{Z})^{t}}.
\end{align}
\end{prop}
\begin{proof}
This follows from \cite[Chapter 7, Corollary 4.41]{liu2}, the exact sequence in Equation (\ref{ExactSequenceTorsion}) and the fact that the reduction map restricted to the $n$-torsion has no kernel for $n$ coprime to $\mathrm{char}(k)$. 
\end{proof}

We can now characterize the torsion points in $J$ using the groups $\pi_{\mathfrak{D}}(\Sigma)$ and $\pi_{\mathfrak{I}}(\Sigma)$. We assume throughout that $n$ is coprime to the residue characteristic. We first note that we have a canonical isomorphism 
\begin{equation}
J[n]\simeq{}{\mathrm{Hom}(\pi(\Sigma),\mathbf{Z}/n\mathbf{Z})}
\end{equation}
by \cite[Chapter III, Lemma 9.2]{Milne1}. That is, there is a natural bijection between cyclic \'{e}tale coverings of $X$ and torsion points of $J$. To be more explicit, let $D\in{J}[n]$ and suppose that $D$ has order $n$. Then $nD=\mathrm{div}(f)$ for some $f\in{K(X)}$ and we consider the covering on the level of function fields defined by
\begin{equation}
z^{n}=f.
\end{equation}
The condition on the order of $D$ ensures that this equation is irreducible and by local considerations, this covering is \'{e}tale. Conversely, every cyclic covering of $X$ is given on the level of function fields by a covering of the form $z^{n}=f$ by Kummer theory. In order for this covering to be \'{e}tale, the valuation of $f$ at every closed point must be divisible by $n$. This then quickly gives an $n$-torsion point in $J$. 

\begin{reptheorem}{TorsionUnramifiedCoverings1}
Let $\mathfrak{I}(\Sigma)$ and $\mathfrak{D}(\Sigma)$ be the inertia and decomposition group of $\Sigma$ in $\pi(\Sigma)$ and let $\pi_{\mathfrak{I}}(\Sigma)$ and $\pi_{\mathfrak{D}}(\Sigma)$ be their corresponding quotients in $\pi(\Sigma)$. Let $n$ be an integer such that $\mathrm{gcd}(n,\mathrm{char}(k))=1$. Then the isomorphism ${J[n]}\simeq{\mathrm{Hom}(\pi(\Sigma),\mathbf{Z}/n\mathbf{Z})}$ induces isomorphisms 
\begin{equation}
J^{0}[n]\simeq{}{\mathrm{Hom}(\pi_{\mathfrak{I}}(\Sigma),\mathbf{Z}/n\mathbf{Z})}
\end{equation}
and
\begin{equation}
T[n]\simeq{}{\mathrm{Hom}(\pi_{\mathfrak{D}}(\Sigma),\mathbf{Z}/n\mathbf{Z})}.
\end{equation} 

\end{reptheorem}
\begin{proof}

Let $\mathcal{X}$ be a strongly semistable model for $X$. On the level of function fields, any cyclic \'{e}tale covering $X'\to{X}$ is given by $K(X')=K(X)(\alpha)$, where $\alpha^{n}=f$ and $f\in{K(X)}$. 
This defines a Galois-topologically tame covering of $X$ and we denote the induced morphism of metrized complexes by $\psi:\Sigma'\rightarrow{\Sigma}$ and the morphism of semistable models by $\mathcal{X}'\rightarrow{\mathcal{X}}$. For any vertex $v$ of $\Sigma(\mathcal{X})$, we denote the corresponding generic point of the special fiber $\mathcal{X}_{s}$ by $\eta_{v}$. The local ring $\mathcal{O}_{\mathcal{X},\eta_{v}}$ is a valuation ring of rank $1$, being locally generated by $\mathfrak{m}_{R}\subset{R}$. We can directly describe the normalization of $\mathcal{X}$ in $K(X')$ 
above this valuation ring as follows. Write $f=\omega^{n}f_{v}$ for some $f_{v}$ with $v_{\eta_{v}}(f_{v})=0$. The element $\alpha'=\alpha/\omega$ is then integral over $\mathcal{O}_{\mathcal{X},\eta}$ as it satisfies $\alpha'^{n}=f_{v}$. Since $\mathrm{gcd}(n,\mathrm{char}(k))=1$, we find that this gives an \'{e}tale extension which thus 
describes the normalization above $\eta_{v}$. That is, the points of $\mathcal{X}'$ that lie above $\eta_{v}$ are described by the \'{e}tale $\mathcal{O}_{\mathcal{X},\eta_{v}}$-scheme 
\begin{equation}
Z=\mathrm{Spec}(\mathcal{O}_{\mathcal{X},\eta_{v}}[z]/(z^{n}-f_{v}))
\end{equation} 
or equivalently by the base change of $Z$ over $\mathrm{Spec}(k(\eta_{v}))$. The closed points of this scheme are as follows. Write $\overline{f}_{v}=f_{1}^{d}$ in $k(\eta_{v})$, where $d$ is the largest divisor of $n$ such that $\overline{f}_{v}$ is a $d$-th power. Let $\zeta$ be a primitive $d$-th root of unity. The factors $z^{n/d}-\zeta^{i}f_{1}$ of $z^{n}-\overline{f}_{v}$ are then irreducible over $k(\eta_{v})$ for every $i\in{0,1,...,d-1}$ by our assumption on $d$. 
We thus find that the polynomials $z^{n/d}-\zeta^{i}f_{1}$ define the extensions of residue fields $k(\eta_{v})\subset{k(\eta_{v_{i}})}$. These are all isomorphic over $k(\eta_{v})$, so it suffices to consider the one determined by $z^{n/d}-f_{1}$. 


We now start with the correspondence for ${J^{0}}$. Let $D\in{J^{0}[n]}$. 
Then $\overline{\tau}(D)=0$, so there is a piecewise linear function $\phi_{D}$ such that $\tau_{*}(D)=\Delta(\phi_{D})$. Here $\Delta(\cdot{})$ is the Laplace operator on piecewise linear functions.  
Let $f\in{K(C)}$ be such that $\mathrm{div}(f)=nD$ and let $\phi_{f}$ be a piecewise linear function on $\Sigma$ such that $\Delta(\phi_{f})=\tau_{*}(\mathrm{div}(f))$. Since the divisors of $n\cdot{\phi_{D}}$ and $\phi_{f}$ are the same, we can use \cite[Theorem 3]{BakerFaber2} and translate these functions so that 
$n\cdot{}\phi_{D}=\phi_{f}$. The slope of $\phi_{f}$ on every edge is then divisible by $n$. 
Let $v$ be a vertex in $\Sigma$ and scale $f$ as in the previous paragraph so that $v_{\eta_{v}}(f_{v})=0$. By the {\it{slope formula}} (see \cite[Theorem 5.15(3)]{BPRa1}) we find that the order of the reduction of $f_{v}$ at the closed point $w_{e}$ in $C_{v}$ corresponding to an edge $e$ is divisible by $n$. Write $\overline{f}_{v}=f_{1}^{d}$ as in the first paragraph and consider a factor $z^{n/d}-f_{1}$ defining the extension of residue curves. 
Since the order of $\overline{f}_{v}$ at the closed point $w_{e}$ is divisible by $n$, we find that the order of $f_{1}$ at $w_{e}$ is divisible by $n/d$. A computation similar to the one in the previous paragraph then shows that the morphisms $C_{v_{i}}\rightarrow{C_{v}}$ are completely split above $w_{e}$ (divide or multiply by a suitable power of a uniformizer at $w_{e}$). We then apply \cite[Theorem 4.23]{ABBR1} and see that $d_{e'}(\psi)=1$ for every edge $e'$ lying above $e$, which shows that $\psi$ is unramified above $e$.

Conversely, suppose that $\psi$ is unramified above every $e$. Let $v$ be a vertex and write $z^{n/d}-f_{1}$ for the polynomial defining an irreducible component above $v$. We claim that the order of $f_{1}$ at any closed point corresponding to an edge is divisible by $n/d$. Indeed, otherwise the Newton polygon of $z^{n/d}-f_{1}$ with respect to the discrete valuation corresponding to the closed point would contain a non-rational slope and thus the extension would be ramified, a contradiction. We conclude by the slope formula that the slope of the piecewise linear function $\phi_{f}$ on every edge $e\in{E(\Sigma)}$ is divisible by $n$. We can then write $\phi_{f}=n\cdot{\phi'}$ for some piecewise linear function $\phi'$. But then $n\Delta(\phi')=\Delta(\phi_{f})=n\tau_{*}(D)$ and thus $\tau_{*}(D)=\Delta(\phi')$. We conclude that $D\in{J^{0}[n]}$, as desired.

Suppose now that $D\in{T}[n]$. We have to check that the induced covering of metrized complexes splits completely. Since $T[n]\subseteq{J^{0}[n]}$, we already know that the covering splits on the edges. We thus only have to show that the covering splits on the vertices. 
By the exact sequence in Equation (\ref{ExactSequenceTorsion}), we know that the reduction of the divisor $D$ at every component is principal, that is $\mathrm{red}(D,\Gamma_{i})=(h_{i})$ for an $h_{i}\in{k(\eta_{v})}$. Furthermore, we have $u_{i}\cdot{}h_{i}^{n}=\overline{f}_{v}$ for some $u_{i}\in{k^{*}}$ by the condition $nD=\mathrm{div}(f)$. It now follows from the explicit description of the normalization given in the first paragraph of the proof that the covering is completely split above every $v$. 


Conversely, suppose that the extension induced by $D\in{J[n]}$ splits above every vertex. Then the reduction of $z^{n}-f_{v}$ 
splits completely at every $v$. This means that $\overline{f}_{v}$ is an $n$-th power in $k(\eta_{v})$.  
We then see that $\mathrm{red}(D,\Gamma_{i})=(h_{i})$ for some $h_{i}\in{k(\eta_{v})}$ and thus the image of $D$ in $\mathrm{Jac}(\mathcal{X}_{s})$ maps to zero in the Jacobian of every component $\Gamma_{i}$. In other words, $D\in{T}[n]$, as desired. 
\end{proof}

\begin{exa}\label{GenusTwoCurveExample1}
Consider a genus two curve $X$ with skeleton as in Figure \ref{PlaatjeGenus2}. To be more explicit, consider the smooth proper curve $X$ defined locally by 
\begin{equation}
y^2=x(x-\varpi)f(x)
\end{equation}
for a polynomial $f(x)\in{R[x]}$ of degree $3$ and an element $\varpi\in{R}$ with $\mathrm{val}(\varpi)>0$. Here we assume that $\mathrm{char}(k)\neq{2}$, $f(0)\neq{0}$ and $f(\varpi)\neq{0}$. 
A simple computation shows that if $\mathrm{deg}(\overline{f}(x))=3$, $\mathrm{gcd}(\overline{f}(x),\overline{f}'(x))=1$ and $\overline{f}(0)\neq{0}$, 
then $X$ has a skeleton of the desired type. 
For any $n$ coprime to $\mathrm{char}(k)$, we have 
\begin{align*}
J[n]&=(\mathbf{Z}/n\mathbf{Z})^{4},\\
J^{0}[n]&=(\mathbf{Z}/n\mathbf{Z})^{3},\\
T[n]&=(\mathbf{Z}/n\mathbf{Z}). 
\end{align*}
 By Theorem \ref{TorsionUnramifiedCoverings1}, 
the coverings that come from $T[n]$ correspond to the topological abelian coverings of $\Sigma$, and the coverings that come from $J^{0}[n]$ correspond to the metrically unramified abelian coverings of $\Sigma$. 
We invite the reader to compare this for $n=2$ with Example \ref{GenusTwoCoverings}. 



\end{exa}

\begin{center}
\bibliographystyle{alpha}
\bibliography{bibfiles2}{}
\end{center}

\end{document}